\documentclass[reqno]{amsart}

\usepackage[all]{xy}
\usepackage{graphicx}

\usepackage{amssymb}
\usepackage{amsmath}
\usepackage{mathrsfs}
\usepackage{epsfig}
\usepackage{amscd}
\usepackage{graphicx, color}

\usepackage{graphicx}
\usepackage{amssymb}
\usepackage{amsmath}
\usepackage{enumerate}
\usepackage{amsfonts}
\usepackage[mathcal]{euscript}
\usepackage{amsthm}
\usepackage{xcolor}
\usepackage{amscd}
\usepackage{mathrsfs}
\numberwithin{equation}{section}

\usepackage[colorlinks=true, pdfstartview=FitV, linkcolor=black,
			   citecolor=black, urlcolor=black]{hyperref}
\usepackage{bookmark}
\usepackage[titletoc]{appendix} 
\hypersetup{hypertexnames=false} 

\theoremstyle{theorem}
\newtheorem{thm}{Theorem}[section]
\newtheorem{cor}[thm]{Corollary}
\newtheorem{lem}[thm]{Lemma}
\newtheorem{prop}[thm]{Proposition}

\theoremstyle{definition}
\newtheorem{defn}[thm]{Definition}
\newtheorem{rem}[thm]{Remark}

\newtheorem{notation}[thm]{\rm\bfseries{Notation}}

\numberwithin{equation}{section}

          \newcommand{\nc}{\newcommand}
          \nc{\DMO}{\DeclareMathOperator}	
          \nc{\commentout}[1]{}

          \nc{\newnotation}{\nomenclature}
          \nc{\wrap}{\cW}
          \nc{\Tw}{\mathsf{Tw}}
          \nc{\loc}{\mathsf{Loc}}
          \nc{\Top}{Top}
          \nc{\emb}{\mathsf{emb}}
          \nc{\ind}{\mathsf{Ind}}
          \nc{\Ind}{\mathsf{Ind}}
          \nc{\Loc}{\mathsf{Loc}}
          \nc{\Cob}{\mathsf{Cob}}
          \nc{\mul}{\mathsf{Mul}}
          \nc{\fat}{\mathsf{fat}}
          \nc{\cob}{\mathsf{Cob}}
          \nc{\coh}{\mathsf{Coh}}
          \nc{\Liouaut}{\Aut_{\mathsf{Liou}}}
          \nc{\idem}{\mathsf{Idem}}
          \nc{\sets}{\mathsf{Sets}}
          \nc{\near}{\mathsf{near}}
          \nc{\sing}{\mathsf{Sing}}
          \nc{\Sing}{\mathsf{Sing}}
          \nc{\perf}{\mathsf{Perf}}
          \nc{\block}{\mathsf{block}}
          \nc{\ssets}{\mathsf{sSets}}
          \nc{\cmpct}{\mathsf{cmpct}}
          \nc{\compact}{\mathsf{cmpct}}
          \nc{\pwrap}{\mathsf{PWrap}}
          \nc{\coder}{\mathsf{Coder}}
          \nc{\bimod}{\mathsf{Bimod}}
          \nc{\grmod}{\mathsf{GrMod}}
          \nc{\Morita}{\mathsf{Morita}}
          \nc{\morita}{\mathsf{Morita}}
          \nc{\spaces}{\mathsf{Spaces}}
          \nc{\pwrms}{\mathsf{PWrFuk}_{M,S}}
          \nc{\pwrmf}{\mathsf{PWrFuk}_{M,F}}
          \nc{\pwrapmf}{\mathsf{PWrFuk}_{M,F}}
          \nc{\fuk}{\mathsf{Fukaya}}
          \nc{\infwr}{\mathsf{InfWr}}
          \nc{\fukaya}{\mathsf{Fukaya}}
          \nc{\autml}{\mathsf{Aut}_{M,\Lambda}}
          \nc{\fukml}{\mathsf{Fukaya}_{M,\Lambda}}
          \nc{\fukmle}{\mathsf{Fukaya}_{M,\Lambda,\epsilon}}
          \nc{\fukmod}{\wrfukcompact(M)\modules}
          \nc{\lag}{\mathsf{Lag}}
          \nc{\lagm}{\lag_M}
          \nc{\lago}{\lag^o}
          \nc{\lagml}{\lag_{M,\Lambda}} 
          \nc{\lagmle}{\lag_{M,\Lambda,\epsilon}}
          \nc{\Fun}{\mathsf{Fun}}
          \nc{\fun}{\mathsf{Fun}}
          \nc{\vect}{\mathsf{Vect}}
          \nc{\chain}{\mathsf{Chain}}
          \nc{\chainn}{Chain}
          \nc{\wrfuk}{\mathsf{WrFukaya}}
          \nc{\wrfukcompact}{\mathsf{WrFukaya}_{\mathsf{cmpct}}}
          \nc{\pwrfuk}{\mathsf{PWrFukaya}}
          \nc{\inffuk}{\mathsf{InfFuk}}
          \nc{\pwrfukml}{\mathsf{PWrFukaya}_{M,\Lambda}}
          \nc{\inffukml}{\mathsf{InfFuk}_{M,\Lambda}}
          \nc{\nattrans}{\mathsf{NatTrans}}
          \nc{\corres}{\mathsf{Corres}}
          \nc{\fukep}{\fukaya_\Lambda(M,\epsilon)}
          \nc{\fukepop}{\fukaya_\Lambda(M,\epsilon)^{\op}}
          \nc{\lagep}{\lag_\Lambda(M,\epsilon)}
          \DMO{\cyl}{cyl} 
          \nc{\dbcoh}{D^b\mathsf{Coh}}
          \nc{\corr}{\mathsf{Corr}}
          \nc{\Liouauto}{{\Aut^o}}
          \nc{\Liouautb}{\Aut^{b}}
          \nc{\Liouautgr}{{\Aut^{gr}}}
          \nc{\Liouautgrb}{\Aut^{gr,b}}
          \nc{\Fuk}{\mathsf{Fuk}}

          \DMO{\im}{im}
          \DMO{\ev}{ev}
          \DMO{\stable}{Ex}
          \DMO{\inj}{inj}
          \DMO{\fib}{fib}
          \DMO{\conf}{Conf}
          \DMO{\chains}{Chains}
          \DMO{\cochains}{Cochains}
          \DMO{\cone}{Cone}
          \DMO{\Map}{Map}
          \DMO{\ran}{Ran}
          \DMO{\rot}{Rot}
          \DMO{\leg}{Leg}
          \DMO{\imm}{imm}
          \DMO{\adj}{adj}
          \DMO{\symp}{Symp}
          \DMO{\tree}{Tree}
          \DMO{\cube}{Cube}
          \DMO{\deep}{deep}
          \DMO{\back}{back}
          \DMO{\Hoch}{Hoch}
          \DMO{\front}{front}
          \DMO{\flow}{Flow}
          \DMO{\floer}{Floer}
          \DMO{\Maps}{Maps}
          \DMO{\exact}{exact}
          \DMO{\excess}{Excess}
          \DMO{\Decomp}{Decomp}
          \DMO{\decomp}{Decomp}
          \DMO{\collar}{collar}
          \DMO{\yoneda}{Yoneda}
          \DMO{\hamspace}{Ham}
          \DMO{\sympspace}{Symp}
          \DMO{\holomaps}{Holomaps}
          \DMO{\comp}{Comp}
          \DMO{\crit}{Crit}
          \DMO{\test}{{test}}
          \DMO{\sign}{sign}
          \DMO{\topp}{top}
          \DMO{\indx}{Index}
          \DMO{\Break}{Break} 
          \DMO{\zero}{zero} 
          \DMO{\ob}{Ob}
          \DMO{\gr}{Gr} 
          \DMO{\Gr}{Gr} 
          \DMO{\cl}{Cl} 
          \DMO{\grlag}{GrLag}
          \DMO{\Pin}{Pin}
          \DMO{\Graph}{Graph}
          \DMO{\pin}{Pin}
          \DMO{\gap}{Gap}
          \DMO{\Ex}{Ex}
          \DMO{\id}{id}
          \DMO{\End}{End}
          \DMO{\sym}{Sym}
          \DMO{\aut}{aut}
          \DMO{\Aut}{Aut}
          \DMO{\haut}{hAut}
          \DMO{\hAut}{hAut}
          \DMO{\DK}{DK} 
          \DMO{\poly}{poly} 
          \DMO{\diff}{Diff}
          \DMO{\coll}{coll}
          \DMO{\dist}{dist} 
          \DMO{\coker}{coker} 
          \nc{\kernel}{\ker} 
          \DMO{\sspan}{span}
          \DMO{\hocolim}{hocolim}	
          \DMO{\holim}{holim}
          \DMO{\sk}{sk}

          \DMO{\ho}{ho}
          \DMO{\fin}{fin}
          \DMO{\tor}{Tor}
          \DMO{\ext}{Ext}
          \DMO{\ret}{Ret}
          \DMO{\ham}{Ham}
          \DMO{\con}{con}
          \DMO{\leaf}{leaf}
          \DMO{\supp}{supp}
          \DMO{\edge}{edge}
          \DMO{\colim}{colim}
          \DMO{\edges}{edges}
          \DMO{\Image}{image}
          \DMO{\roots}{roots}
          \DMO{\height}{height}
          \DMO{\finmod}{FinMod}
          \DMO{\leaves}{leaves}
          \DMO{\planar}{planar}
          \DMO{\vertices}{vertices}

\nc{\norm}[2]{{ \ensuremath{\|} #1 \ensuremath{\|}}_{#2}}
\nc{\Dbar}[1]{\ensuremath{{\bar{\partial}}_{#1}}}
\nc{\Ce}{\ensuremath{\mathbb{C}}}
\nc{\B}{\ensuremath{\mathbb{B}}}
\nc{\osc}{\operatorname{osc}}
\nc{\leng}{\operatorname{leng}}

          \nc{\lagg}{\lag^{\cG}}
          \nc{\iso}{\mathsf{Iso}}
          \nc{\Set}{\mathsf{Set}}
          \nc{\Ass}{\mathsf{ \bf Ass}}
          \nc{\Mod}{\mathsf{Mod}}
          \nc{\modules}{\mathsf{Mod}}
          \nc{\sset}{\mathsf{sSet}}
          \nc{\liou}{\mathsf{Liou}}
          \nc{\poset}{\mathsf{Poset}}
          \nc{\trno}{T^*\RR^n_{\geq 0}}
          \nc{\spectra}{\mathsf{Spectra}}
          \nc{\tensorfin}{\tensor^{\fin}}
          \nc{\lagptg}{\lag_{pt,pt}^{\cG}}
          \nc{\Fin}{\mathcal{F}\mathsf{in}}
          \nc{\lagnl}{\lag_{N,\Lambda}}
          \nc{\lagmlg}{\lag_{M,\Lambda}^{\cG}}
          \nc{\lagsplit}{\lag^{\mathsf{split}}}
          \nc{\lagktimes}{(\lag^{\dd k})^\times}
          \nc{\lagplanar}{\lag^{\times,\planar}}
          \nc{\Cont}{\text{\rm Cont}}
          \nc{\Ham}{\text{\rm Ham}}
          \nc{\Dev}{\text{\rm Dev}}
          \nc{\Lin}{\text{\rm Lin}}
          \nc{\Int}{\text{\rm Int}}
          \nc{\Hom}{\text{\rm Hom}}
          \nc{\Chord}{\text{\rm Chord}}
          \nc{\nbhd}{\mathcal{N}\text{\rm{bhd}}}

          \nc{\onef}{1_{\fukaya}}
          \nc{\smsh}{\wedge}
          \nc{\un}{\underline}
          \nc{\xto}{\xrightarrow}
          \nc{\xra}{\xto}
          \nc{\tensor}{\otimes}
          \nc{\del}{\partial}
          \nc{\dd}{\diamond}
          \nc{\tri}{\triangle}
          \nc{\CB}{\Box}
          \nc{\into}{\hookrightarrow}
          \nc{\onto}{\twoheadrightarrow}
          \nc{\contains}{\supset}
          \nc{\transverse}{\pitchfork}
          \nc{\uncirc}{\underline{\circ}}
          \nc{\thetacontact}{\theta} 

          \nc{\Jbar}{\overline{J}}
          \nc{\Fbar}{\overline{F}}
          \nc{\delbar}{\overline{\del}}
          \nc{\thetabar}{\overline{\theta}}
          \nc{\omegabar}{\overline{\omega}}
          \nc{\Liou}{\text{\rm Liou}}

          \nc{\Yhat}{\widehat{Y}}

          \nc{\Mliou}{M}

          \nc{\vece}{ {\vec \epsilon}}	
          \nc{\vecd}{ {\vec \delta}}
          \nc{\ov}{\overline}
          \DMO{\op}{op}
          \nc{\opp}{ ^{\op}}

          \nc{\hiro}{\textcolor{blue}}
          \nc{\YG}{\textcolor{orange}}
          \nc{\beastar}{\begin{eqnarray*}}
          \nc{\eeastar}{\end{eqnarray*}}

\numberwithin{equation}{section}

\def\R{{\mathbb R}}
\def\osc{{\hbox{\rm osc }}}
\def\Crit{{\hbox{Crit}}}

\def\E{{\mathbb E}}
\def\Z{{\mathbb Z}}
\def\C{{\mathbb C}}
\def\R{{\mathbb R}}
\def\P{{\mathbb P}}

\def\N{{\mathbb N}}

\def\11{{\mathbb I}}

\def\Jbar{{\widetilde J}}

\def\delbar{{\overline \partial}}

\def\C{\mathbb{C}}
\def\Z{\mathbb{Z}}

\def\Q{\mathbb{Q}}

\def\E{\ifmmode{\mathbb E}\else{$\mathbb E$}\fi} 
\def\N{\ifmmode{\mathbb N}\else{$\mathbb N$}\fi} 
\def\R{\ifmmode{\mathbb R}\else{$\mathbb R$}\fi} 
\def\Q{\ifmmode{\mathbb Q}\else{$\mathbb Q$}\fi} 
\def\C{\ifmmode{\mathbb C}\else{$\mathbb C$}\fi} 
\def\H{\ifmmode{\mathbb H}\else{$\mathbb H$}\fi} 
\def\Z{\ifmmode{\mathbb Z}\else{$\mathbb Z$}\fi} 
\def\P{\ifmmode{\mathbb P}\else{$\mathbb P$}\fi} 
\def\SS{\ifmmode{\mathbb S}\else{$\mathbb S$}\fi} 
\def\DD{\ifmmode{\mathbb D}\else{$\mathbb D$}\fi} 

\def\R{{\mathbb R}}
\def\osc{{\hbox{\rm osc}}}
\def\Crit{{\hbox{Crit}}}
\def\E{{\mathbb E}}
\def\Z{{\mathbb Z}}
\def\C{{\mathbb C}}
\def\R{{\mathbb R}}

\def\N{{\mathbb N}}

\def\CJ{{\mathcal J}}

\def\FF{{\mathcal F}}

\def\Jbar{{\widetilde J}}

\def\delbar{{\overline \partial}}

\def\Ev{\text{\rm Ev}}

\def\b{\beta}

\def\e{\varepsilon}

\def\o{\omega}  

  \def\S{\Sigma}

\def\w{\varphi}

\def\CB{{\mathcal B}}
\def\CC{{\mathcal C}}
\def\CD{{\mathcal D}}

\def\CF{{\mathcal F}}

\def\CH{{\mathcal H}}

\def\CJ{{\mathcal J}}

\def\CL{{\mathcal L}}
\def\CM{{\mathcal M}}

\def\CW{{\mathcal W}}

\def\darr#1{\raise1.5ex\hbox{$\leftrightarrow$}
\mkern-16.5mu #1}

\def\roughly#1{\raise.3ex\hbox{$#1$\kern-.75em
\lower1ex\hbox{$\sim$}}}

\def\opname#1{\mathop{\kern0pt{\rm #1}}\nolimits}

\def\End{\opname{End}}

\def\dist{\opname{dist}}

\def\rank{\opname{rank}}

\def\supp{\operatorname{supp}}

\def\Dev{\operatorname{Dev}}

\def\leng{\operatorname{leng}}

\def\End{\operatorname{End}}

\def\Aut{\operatorname{Aut}}
\def\coker{\operatorname{Coker}}
\def\span{\operatorname{span}}

\def\Cont{\operatorname{Cont}}
\def\Crit{\operatorname{Crit}}
\def\Spec{\operatorname{Spec}}
\def\Sing{\operatorname{Sing}}
\def\GFQI{\frak{G}}
\def\Index{\operatorname{Index}}

\def\Image{\operatorname{Image}}
\def\ev{\operatorname{ev}}
\def\Int{\operatorname{Int}}

\def\ben{\begin{enumerate}}
\def\een{\end{enumerate}}
\def\be{\begin{equation}}
\def\ee{\end{equation}}
\def\bea{\begin{eqnarray}}
\def\eea{\end{eqnarray}}
\def\beastar{\begin{eqnarray*}}
\def\eeastar{\end{eqnarray*}}
\def\bc{\begin{center}}
\def\ec{\end{center}}

\renewcommand{\b}{\beta}

\begin{document}

\quad \vskip1.375truein

\def\mq{\mathfrak{q}}
\def\mp{\mathfrak{p}}
\def\mH{\mathfrak{H}}
\def\mh{\mathfrak{h}}
\def\ma{\mathfrak{a}}
\def\ms{\mathfrak{s}}
\def\mm\mm{\mathfrak{m}}
\def\mn{\mathfrak{n}}
\def\mz{\mathfrak{z}}
\def\mw{\mathfrak{w}}
\def\Hoch{{\tt Hoch}}
\def\mt{\mathfrak{t}}
\def\ml{\mathfrak{l}}
\def\mT{\mathfrak{T}}
\def\mL{\mathfrak{L}}
\def\mg{\mathfrak{g}}
\def\md{\mathfrak{d}}
\def\mr{\mathfrak{r}}
\def\Cont{\operatorname{Cont}}
\def\Crit{\operatorname{Crit}}
\def\Spec{\operatorname{Spec}}
\def\Sing{\operatorname{Sing}}
\def\GFQI{\text{\rm g.f.q.i.}}
\def\Index{\operatorname{Index}}
\def\Cross{\operatorname{Cross}}
\def\Ham{\operatorname{Ham}}
\def\Fix{\operatorname{Fix}}
\def\Graph{\operatorname{Graph}}
\def\id{\text\rm{id}}

\title[Evaluation transversality of contact instantons]
{Generic jet evaluation transversality of contact instantons against contact distribution}

\author{Yong-Geun Oh}
\address{Center for Geometry and Physics, Institute for Basic Science (IBS),
77 Cheongam-ro, Nam-gu, Pohang-si, Gyeongsangbuk-do, Korea 790-784
\& POSTECH, Gyeongsangbuk-do, Korea}
\email{yongoh1@postech.ac.kr}
\thanks{This work is supported by the IBS project \# IBS-R003-D1}


\begin{abstract} For a given coorientable contact manifold $(M,\Xi)$ with contact distribution $\Xi$, 
we consider its contact forms $\lambda$ with $\ker \lambda = \Xi$, and the associated contact triads
$(M,\lambda, J)$. For a generic choice of contact form $\lambda$, 
we  prove the (0-jet)  the interior and boundary evaluation maps,
and  the 1-jet transversality of contact instantons (against contact distribution, for example).
\end{abstract}


\maketitle

\tableofcontents

\section{Introduction}

It has been well-known that the evaluation transversality is one of the fundamental 
problem together with the mapping transversality in the Gromov's theory of pseudoholomorphic
curves in symplectic topology and in the Gromov-Witten theory. The higher jet-evaluation
transversality problem has been systematically studied by the present author in \cite{oh:higher-jet}
in his attempt to establish the symplectic version of the \emph{ramified Gromov-Witten invariants}
defined by Kim, Kresch and the present author in algebraic geometry \cite{kim-kresch-oh}.
More recently, Wendl  beautifully applied this line of study in his proof of
the super rigidity of embedded pseudoholomorphic curves \cite{wendl:super-rigidity}.

In the present paper, we prove the generic 1-jet evaluation transversality of the moduli space
of contact instantons and its Hamiltonian-perturbed counterpart, which is one of the essential
ingredients needed in our proof of Weinstien's conjecture \cite{oh:weinstein-conjecture}.
The scheme of our proof is an enhancement of the \emph{canonical scheme} of
the  proof of generic $0$-jet evaluation transversality utilized in \cite{oh:contacton-transversality}
which is an adaptation of the canonical scheme utilized  in \cite{oh-zhu:ajm}, \cite{oh:higher-jet}
for the study of pseudoholomorphic curves in symplectic geometry.

In  \cite{oh:contacton}, the present author  initiated  the
Fredholm theory of contact instantons by deriving an explicit tensorial formula of the linearized operator
of the contact instantons, denoted by  $D\Upsilon(w)$  from now on, in \cite[Theorem 10.1]{oh:contacton}.
(See also  Theorem \ref{thm:linearization} in the present paper.)

We note that there are three kinds of perturbations we can think of as mentioned above, i.e., $J$, 
$\lambda$ and the boundary Legendrian submanifolds $R_i$'s. The study of perturbations of $J$ is given in
\cite{oh:contacton} which is of no change in the present open string case, and perturbation of
contact forms is largely subsumed into that of $J$. The Fredholm
theory for the bordered contact instantons has been then given in \cite{oh:contacton-transversality}.

We now formulate the precise statement of the 1-jet transversality of the derivative 
of the evaluation map $\ev^+:\widetilde \CM_{(0,1)} (\dot \Sigma,M;J) \to M$ given by
$\ev^+(w,z^+) = w(z^+)$ where the subindex $(k,\ell)$ stands for the number of 
boundary marked points $\{z_1, \cdots, z_k\}$ and that of interior marked points 
$\{z_1^+, \cdots, z_\ell^+\}$, respectively.
We will be particularly interested in the transversality result against the manifold
$$
\Xi \subset TM
$$
of the contact distributions. For the study, the following reformulation of this transversality
when a contact form $\lambda$ turns out to be useful.

\begin{lem}\label{lem:lambda-transversality}
Let $\lambda$ be a contact form of a contact manifold $(M,\Xi)$.
Consider $\varphi: N \to M$ be any smooth map. Then the differential
$$
d\varphi: TN \to TM
$$
is transversal to $\Xi$ if and only if the function $TN \to \R$
defined by $(x,v) \mapsto \lambda_x(d\varphi(v))$ has 0 as a regular value.
\end{lem}

The following moduli space with 1-jet constraint will play a crucial role in our
proofs of Weinstein' conjecture and Arnold's chord conjecture \cite{oh:weinstein-conjecture}.

\begin{defn}[Moduli space with $\Xi$-tangency condition]
We call  the 1-marked moduli space
$$
\CM_{(1,0)}^\Xi(\dot \Sigma, M;J)
= \left\{(w,z^+) \, \Big|\,  w \in \CM_{(0,1);[0,K_0]}(\dot \Sigma, M;J), 
\lambda(dw(z^+)) = 0, \, z^+ \in \dot \Sigma \right\}
$$
\emph{the moduli space with $\Xi$-tangency condition}.
\end{defn}

We take the union of  moduli spaces 
$$
\widetilde \CM_{(0,1)}^\Xi(\dot \Sigma, M) = \bigcup_{J \in \CJ_\lambda}
\{J\} \times \widetilde \CM_{(0,1)}^\Xi(\dot \Sigma,M;J)
$$
where $\operatorname{Aut}(\Sigma)$ acts on $((\Sigma,j),w)$ by conformal
reparameterization for any $j$.

The following is the main theorem of the present paper.

\begin{thm}[$1$-jet evaluation transversality]\label{thm:1-jet}
The derivative
$$
D(\ev^+):T \widetilde \CM_{(0,1)}(\dot \Sigma,M) \to \Lambda^1((\cdot)^*\Xi)
$$
is transverse to $\Lambda^1(w^*\Xi)$, where its fiberwise restriction
$$
D_{{(J,(j,w),z)} }(\ev^+): T_{(J,(j,w),z)} \widetilde \CM_{(0,1)}(\dot \Sigma,M) \to \Lambda_z^1(w^*TM)
$$
is the linear map defined by
\be\label{eq:Dev+}
D_{{(J,(j,w),z)} }(\ev^+): T_{(J,(j,w),z)}(B, (a,X), v): = \nabla_{dw(z^+)(v)} X + \nabla_v dw.
\ee
at each $(J,(j,w),z) \in \widetilde \CM_{(0,1)}(\dot \Sigma,M)$.
\end{thm}

The following corollary will be an important ingredient in \cite{oh:weinstein-conjecture}.

\begin{cor}\label{cor:MM-Xi-tangency} Let $\dot \Sigma = \R \times [0,1]$ and wrote
$$
D_t(\ev^+) : T \widetilde \CM_{(0,1)}(\dot \Sigma,M) \to TM
$$
defined by 
$$
D_t(\ev^+)(w,z^+) = dw\left(\frac{\del}{\del t}\right) = \frac{\del w}{\del t}\left(\tau, t\right).
$$
Then for a generic choice of $J$, the moduli space
\beastar 
&{}& 
(D_t(\ev^+))^{-1}(\Xi)  \\
& =  & \Big \{(w,z^+) \, \Big|\,  w \in \CM_{[0,K_0]}(\dot \Sigma, M;J), 
\lambda(dw(z^+)) \left(\frac{\del}{\del t} \right) = 0, \,
 \, z^+ = (\tau,t), \, \tau \in \R  \Big\} 
\eeastar
is a smooth manifold of dimension equal to that of $\CM_{[0,K_0]}(\dot \Sigma, M;J)$.
\end{cor}
We write
\be\label{eq:MM-Xi}
\widetilde \CM_{(0,1)}^{\Xi,t} (\dot \Sigma,M;J): = (D_t(\ev^+))^{-1}(\Xi).
\ee
Based on the study of higher jet evaluation transversality established in \cite{oh:higher-jet}
(also in \cite{wendl:super-rigidity}) for the pseudoholomorphic curve
moduli spaces in symplectic geometry, similar higher jet evaluation transversality can be also
established which we omit and leave to the interested readers.

\section{Preliminaries}

For the purpose of studying the generic evaluation transversality, similarly as in symplectic geometry
\cite{oh:fredholm}, we consider the universal moduli space
$$
\CM^{\text{\rm univ}}(M,\lambda;\overline \gamma, \underline \gamma)
$$
consisting of the triples $(w,(J,\vec R))$ satisfying 
\be\label{eq:contacton-Legendrian-bdy-intro}
\begin{cases}
\delbar^\pi_J w = 0, \, \quad d(w^*\lambda \circ j) = 0\\
w(\overline{z_iz_{i+1}}) \subset R_i, \quad i = 1, \ldots, k
\end{cases}
\ee
for a map $w:(\dot \Sigma, \del \dot \Sigma) \to (M,\vec R)$ with the boundary
condition 
with $\vec R = \{R_i\}_{i=1}^k$ \emph{with fixed asymptotics
$(\overline \gamma, \underline \gamma)$ at the punctures of $\dot \Sigma$}.
Denote by $\mathfrak{Led}(M,\xi)$ the set of smooth Legendrian
submanifolds with this given set of Reeb chords as its asymptotics boundary condition.

We regard the assignment
$$
\Upsilon^{\text{\rm univ}}: w \mapsto \left(\delbar^\pi_J w, d(w^*\lambda \circ j)\right), \quad
\Upsilon: = (\Upsilon_1,\Upsilon_2)
$$
as a section of the (infinite dimensional) vector bundle
\be\label{eq:universal-CD}
\CC\CD^{\text{\rm univ}} \to \CF^{\text{\rm univ}}(M,\lambda; \overline \gamma, \underline \gamma)
\ee
where we put
$$
\CF^{\text{\rm univ}}(M,\lambda; \overline \gamma, \underline \gamma): =
\bigcup_{\vec R \in \mathfrak{Led}(M,\xi)} \CF\left(M,\lambda;\vec R; 
\overline \gamma, \underline \gamma\right)
\times  \CJ_\lambda(M,\xi).
$$
We denote by $\Upsilon^{\text{\rm univ}}$ the parameterized section of \eqref{eq:universal-CD}
defined by
$$
\Upsilon^{\text{\rm univ}}((w,J),\vec R) = \left(\delbar^\pi_J w, d(w^*\lambda \circ j)\right).
$$
Here $\CF(M,\lambda; \vec R; \overline \gamma, \underline \gamma)$ is the off-shell function space
associated to the moduli space $\CM(M,\lambda;J;\vec R; \overline \gamma, \underline \gamma)$.
(See Definition \ref{defn:Banach-manifold}.)
The relevant Fredholm result has been
established in \cite{oh:contacton} for the closed string case and
for the open string case in \cite{oh:contacton-transversality}.

We start with recalling the framework for the study of generic
nondegeneracy results for the Reeb orbits and chords given in 
\cite{oh-yso:index}, \\cite{oh:contacton-transversality}.

\begin{defn} \label{defn:Moore-path2}Let $T \geq 0$ and consider a curve $\gamma:[0,1] \to M$ be a
smooth curve. We say $(\gamma,T)$ an \emph{iso-speed  Reeb trajectory} if the pair satisfies
$$
\dot \gamma(t) = T R_\lambda(\gamma(t)), \, \int \gamma^*\lambda = T
$$
for all $t \in [0,1]$. If $\gamma(1) = \gamma(0)$, we call $(\gamma,T)$ an
iso-speed closed Reeb orbit and $T$ the \emph{action} of $\gamma$.
\end{defn}

We start with the case of closed orbits.

\begin{defn} Let $(\gamma,T)$ be an iso-speed closed Reeb orbit in the sense as above.
When $|T| > 0$ is minimal among such that $\gamma(1) = \gamma(0)$ with $\int \gamma^*\lambda \neq 0$,
we call the pair $(\gamma,T)$ a \emph{simple} iso-speed closed Reed orbit.
\end{defn}

We consider the relative version thereof.

\begin{defn}[Iso-speed Reeb chords \cite{oh:entanglement1}]
 Let $(R_0,R_1)$ be a pair of Legendrian submanifolds of $(M,\xi)$ and $T\geq 0$.
For given contact form $\lambda$, we say $(\gamma, T)$ with $\gamma:[0,1] \to M$
is a Reeb chord from $R_0$ to $R_1$
if
$$
\dot \gamma(t) = T \dot R_\lambda(\gamma(t)), \quad \gamma(0) \in R_0, \, \gamma(1) \in R_1.
$$
We call any such $(\gamma, T)$ an \emph{iso-speed Reeb chord} and
\emph{nonnegative} if $T \geq 0$.
\end{defn}

Let $(\gamma, T)$ be a closed Reeb orbit of action $T$.
By definition, we can write $\gamma(T) = \phi^T_{R_\lambda}(\gamma(0))$
for the Reeb flow $\phi^T= \phi^T_{R_\lambda}$ of the Reeb vector field $R_\lambda$.
In particular $p = \gamma(0)$ is a fixed point of the diffeomorphism $\phi^T$.
Since $L_{R_\lambda}\lambda = 0$, $\phi^T$ is a contact diffeomorphism and so
induces an isomorphism
$$
\Psi_\gamma : = d\phi^T(p)|_{\xi_p}: \xi_p \to \xi_p
$$
which is the linearization restricted to $\xi_p$ of the Poincar\'e return map.

\begin{defn} Let $T> 0$. We say a $T$-closed Reeb orbit $(T,\lambda)$ is \emph{nondegenerate}
if $\Psi_\gamma:\xi_p \to \xi_p$ with $p = \gamma(0)$ has not eigenvalue 1.
\end{defn}

The  following generic nondegeneracy result is proved in \cite{oh:contacton-transversality}
for Reeb chords which extends the above generic nondegeneracy results to
the case of open strings of Reeb chord and to the Bott-Morse situation of constant chords.

\begin{thm} \label{thm:Reeb-chords}
Let $(M,\xi)$ be a contact manifold. Let  $(R_0,R_1)$ be a pair of Legendrian submanifolds
allowing the case $R_0 = R_1$.
\begin{enumerate}
\item For a given pair $(R_0,R_1)$, there exists a residual subset
$$
\CC^{\text{\rm reg}}(\xi;R_0,R_1) \subset \CC(M,\xi)
$$
such that for any $\lambda \in \CC^{\text{\rm reg}}(\xi;R_0,R_1) $ all
Reeb chords from $R_0$ to $R_1$ are nondegenerate for $T > 0$ and
Bott-Morse nondegenerate when $T = 0$.
\item For a given contact form $\lambda$, there exists a residual subset of pairs $(R_0,R_1)$
of Legendrian submanifolds such that  all Reeb chords from $R_0$ to $R_1$ are
nondegenerate for $T > 0$ and
Bott-Morse nondegenerate when $T = 0$.
\item When $n \geq 2$ and when $R_0=R_1 = R$, there is no closed Reeb orbit of $(M,\lambda)$
intersecting $R$.
\end{enumerate}
\end{thm}

\section{Off-shell Fredholm setting for the evaluation transversality}
\label{sec:ev-transversality}

We first recall the off-shell setting of the study of linearized operator 
the moduli space
$$
\CM_k((\dot \Sigma,\del \dot \Sigma),(M,\vec R)), \quad \vec R = (R_1,\cdots, R_k)
$$
of finite energy contact instantons $w: \dot \Sigma \to M$ satisfying the equation
\eqref{eq:contacton-Legendrian-bdy-intro} as before.

For the review of the Fredholm setting, although we will need the study only of 
interior evaluation maps for the application to the orderability study, we will consider the general setting of 
contact instantons of arbitrary genus including the Legendrian boundary condition for the future purpose,
because it is this general setting employed in the present author's article 
\cite{oh:contacton-transversality} and we would like to hint that our exposition will also
apply to the boundary evaluation map. Consideration of this general case does not make much difference except 
incorporation of boundary condition since we will only state the relevant statements
without mentioning their proofs in this review. We refer to ibid. for the readers
interested in looking at details of the proofs thereof.

\subsection{Off-shell description of moduli spaces}

{We borrow the exposition given in \cite{oh:contacton-transversality} and \cite{oh-yso:index} here.

Let $(\Sigma, j)$ be a bordered compact Riemann surface, and let $\dot \Sigma$ be the
punctured Riemann surface with $\{z_1,\ldots, z_k \} \subset \del \Sigma$, we consider the moduli space
$$
\CM_k((\dot \Sigma,\del \dot \Sigma),(M,\vec R)), \quad \vec R = (R_1,\cdots, R_k)
$$
 of finite energy maps $w: \dot \Sigma \to M$ satisfying the equation \eqref{eq:contacton-Legendrian-bdy-intro}.

We will be mainly interested in the two cases:
\begin{enumerate}
\item A generic nondegenerate case of $R_1, \cdots, R_k$ which in particular
are mutually disjoint,
\item The case where $R_1, \cdots, R_k = R$.
\end{enumerate}

The  second case is transversal in the Bott-Morse sense both for the Reeb
chords and for the moduli space of contact instantons, which is
rather straightforward and easier to handle, and so omitted.

For the first case, all the asymptotic Reeb chords  are nonconstant and have nonzero
action $T \neq 0$. In particular, the relevant punctures $\chi^+
$
are not removable. Therefore we have the decomposition of the finite energy moduli space
$$
\CM_k((\dot \Sigma,\del \dot \Sigma),(M,\vec R)) =
\bigcup_{\vec \gamma \in \prod_{i=0}^{k-1}\frak{Reeb}(R_i,R_{i+1})} \CM(\vec \gamma),
\quad \gamma = (\gamma_1, \ldots, \gamma_k).
$$
Depending on the choice of strip-like coordinates we divide the punctures
$$
\{z_1, \cdots, z_k\} \subset \del \Sigma
$$
into two subclasses
$$
p_1, \cdots, p_{s^+}, q_1, \cdots, q_{s^-} \in \del \Sigma
$$
as the positive and negative boundary punctures. We write $k = s^+ + s^-$.

Let $\gamma^+_i$ for $i =1, \cdots, s^+$ and $\gamma^-_j$ for $j = 1, \cdots, s^-$
be two given collections of Reeb chords at positive and negative punctures
respectively.  We denote by $\underline \gamma$ and $\overline \gamma$ the corresponding
collections
\beastar
\underline \gamma & = & \{\gamma_1^+,\cdots, \gamma_{s^+}\} \\
\overline \gamma & = & \{\gamma_1^+,\cdots, \gamma_{s^-}\}.
\eeastar
For each $p_i$ (resp. $q_j$), we associate the strip-like
coordinates $(\tau,t) \in [0,\infty) \times S^1$ (resp. $(\tau,t) \in (-\infty,0] \times S^1$)
on the punctured disc $D_{e^{-2\pi K_0}}(p_i) \setminus \{p_i\}$
(resp. on $D_{e^{-2\pi R_0}}(q_i) \setminus \{q_i\}$) for some sufficiently large $K_0 > 0$.

\begin{defn}\label{defn:Banach-manifold} We define
\be\label{eq:offshell-space}
\CF((\dot \Sigma, \del \dot \Sigma),(M, \vec R);J;\underline \gamma,\overline \gamma)
\ee
to the set of smooth maps satisfying the boundary condition
\be\label{eq:bdy-condition}
w(z) \in R_i \quad \text{ for } \, z \in \overline{z_{i-1}z_i}
 \subset \del \dot \Sigma
\ee
and the asymptotic condition
\be\label{eq:limatinfty}
\lim_{\tau \to \infty}w((\tau,t)_i) = \gamma^+_i(T_i(t+t_i)), \qquad
\lim_{\tau \to - \infty}w((\tau,t)_j) = \overline \gamma_j(T_j(t-t_j))
\ee
for some $t_i, \, t_j \in S^1$, where
$$
T_i = \int_{S^1} (\gamma^+_i)^*\lambda, \, T_j = \int_{S^1} ( \gamma^-_j)^*\lambda.
$$
Here $t_i,\, t_j$ depends on the given analytic coordinate and the parameterizations of
the Reeb chords.
\end{defn}
We will fix $j$ and its associated K\"ahler metric $h$.
We regard the assignment
$$
\Upsilon: w \mapsto \left(\delbar^\pi w, d(w^*\lambda \circ j)\right), \quad
\Upsilon: = (\Upsilon_1,\Upsilon_2)
$$
as a section of the (infinite dimensional) vector bundle:
We first formally linearize and define a linear map
$$
D\Upsilon(w): \Omega^0(w^*TM,(\del w)^*T\vec R) \to \Omega^{(0,1)}(w^*\Xi) \oplus \Omega^2(\Sigma)
$$
where we have the tangent space
$$
T_w \CF = \Omega^0\left(w^*TM,(\del w)^*T\vec R\right).
$$
For the optimal expression of the linearization map and its relevant
calculations, we use the contact triad connection $\nabla$ of $(M,\lambda,J)$ and the contact
Hermitian connection $\nabla^\pi$ for $(\xi,J)$ introduced in
\cite{oh-wang:connection,oh-wang:CR-map1}.

Let $k \geq 2$ and $p > 2$. We denote by
\be\label{eq:CWkp}
\CW^{k,p}: = \CW^{k,p}\left((\dot \Sigma, \del \dot \Sigma),(M, \vec R); \underline \gamma,\overline \gamma\right)
\ee
the completion of the space \eqref{eq:offshell-space}.
It has the structure of a Banach manifold modelled by the Banach space given by the following

\begin{defn}\label{defn:tangent-space} We define
$$
W^{k,p} \left(w^*TM, (\del w)^*T\vec R; \underline \gamma,\overline \gamma\right)
$$
to be the set of vector fields $Y = Y^\pi + \lambda(Y) R_\lambda$ along $w$ that satisfy
\be\label{eq:tangent-element}
\begin{cases}
Y^\pi \in W^{k,p}\left((\dot\Sigma, \del \dot \Sigma), \left(w^*\Xi, (\del w)^*T\vec R\right)\right),  \\
\lambda(Y) \in W^{k,p}((\dot \Sigma, \del \dot \Sigma),(\R, \{0\})),\\
Y^\pi(\del \dot \Sigma) \subset T\vec R, \qquad \lambda(Y)(\del \dot \Sigma) = 0.
\end{cases}
\ee
\end{defn}
Here we use the splitting
$$
TM = \xi \oplus \span_\R\{R_\lambda\}
$$
where $\span_\R\{R_\lambda\}: = \CL$ is a trivial line bundle and so
$$
\Gamma(w^*\CL) \cong C^\infty(\dot \Sigma, \del \dot \Sigma).
$$
The above Banach space is decomposed into the direct sum
\be\label{eq:tangentspace}
W^{k,p}\left((\dot\Sigma,\del \dot \Sigma),( w^*\Xi, (\del w)^*T\vec R)\right)
\bigoplus W^{k,p}((\dot \Sigma,\del \dot \Sigma), ( \R, \{0\})) \otimes R_\lambda :
\ee
by writing $Y = (Y^\pi, g R_\lambda)$ with a real-valued function $g = \lambda(Y(w))$ on $\dot \Sigma$.
Here we measure the various norms in terms of the triad metric of the triad $(M,\lambda,J)$.

Now for each given $w \in \CW^{k,p}((\dot \Sigma,\del \dot \Sigma), (M, \vec R);J;\underline \gamma,\overline \gamma)$,
we consider the Banach space
$$
\Omega^{(0,1)}_{k-1,p}(w^*\Xi): = W^{k-1,p}\left(\Lambda_{j,J}^{(0,1)}(w^*\Xi)\right)
$$
the $W^{k-1,p} $-completion of $\Omega^{(0,1)}(w^*\Xi) = \Gamma(\Lambda^{(0,1)}(w^*\Xi))$ and form the bundle
\be\label{eq:CH01}
\bigcup_{w \in \CW^{k,p}} \Omega^{(0,1)}_{k-1,p}(w^*\Xi)
\ee
over $\CW^{k,p}$.

\begin{defn}\label{defn:CHCM01} We associate the Banach space
\be\label{eq:CH01-w}
\CH^{\pi(0,1)}_{k-1,p}(M,\lambda)|_w: = \Omega^{(0,1)}_{k-1,p}(w^*\Xi) \oplus \Omega^2_{k-2,p}(\dot \Sigma)
\ee
to each $w \in \CW^{k,p}$ and form the bundle
$$
\CH^{\pi(0,1)}_{k-1,p}(M,\lambda): = \bigcup_{w \in \CW^{k,p}} \{w\} \times \CH^{\pi(0,1)}_{k-1,p}(M,\lambda)|_w
$$
over $\CW^{k,p}$.
\end{defn}

Then we can regard the assignment
$$
\Upsilon_1: w \mapsto \delbar^\pi w
$$
as a smooth section of the bundle $\CH^{\pi(0,1)}_{k-1,p}(M,\lambda) \to \CW^{k,p}$. Furthermore
the assignment
$$
\Upsilon_2: w \mapsto d(w^*\lambda \circ j)
$$
defines a smooth section of the trivial bundle
$$
\Omega^2_{k-2,p}(\Sigma) \times \CW^{k,p} \to \CW^{k,p}
$$
for the Banach manifold
$$
\CW^{k,p}:= \CW^{k,p}\left((\dot \Sigma,\del \dot \Sigma),(M, \vec R);J;\underline \gamma,\overline \gamma\right).
$$
We summarize the above discussion to the following lemma.

\begin{lem}\label{lem:Upsilon} Consider the vector bundle
$$
\CH^{\pi(0,1)}_{k-1,p}(M,\lambda) \to \CW^{k,p}.
$$
The map $\Upsilon$ continuously extends to a continuous section
$$
\Upsilon: \CW^{k,p} \to \CH^{\pi(0,1)}_{k-1,p} (\xi;\vec R).
$$
\end{lem}

With these preparations, the following is a consequence of the exponential estimates established
in \cite{oh-wang:CR-map1} for the case of vanishing charge.

\begin{prop}[Theorem 1.12 \cite{oh-wang:CR-map1}]
Assume $\lambda$ is nondegenerate and $Q(p_i) = 0$.
Let $w:\dot \Sigma \to M$ be a contact instanton and let $w^*\lambda = a_1\, d\tau + a_2\, dt$.
Suppose
\bea
\lim_{\tau \to \infty} a_{1,i} = 0, &{}& \, \lim_{\tau \to \infty} a_{2,i} = T(p_i)\nonumber\\
\lim_{\tau \to -\infty} a_{1,j} = 0, &{}& \, \lim_{\tau \to -\infty} a_{2,j} = T(p_j)
\eea
at each puncture $p_i$ and $q_j$.
Then $w \in \CW^{k,p}(\dot \Sigma, M;J;\underline \gamma,\overline \gamma)$.
\end{prop}

Now we are ready to define the moduli space of contact instantons with prescribed
asymptotic condition.
\begin{defn}\label{defn:tilde-modulispace} Consider the zero set of the section $\Upsilon$
\be\label{eq:defn-tildeMM}
\widetilde \CM\left((\dot \Sigma,\del \dot \Sigma),(M,\vec R);J;\underline \gamma,\overline \gamma\right) =  \Upsilon^{-1}(0)
\ee
in the Banach manifold $\CW^{k,p}\left((\dot \Sigma,\del \dot \Sigma),(M,\vec R);J;\underline \gamma,\overline \gamma\right)$, and
\be\label{eq:defn-MM}
\widetilde \CM\left((\dot \Sigma,\del \dot \Sigma),(M,\vec R);J;\underline \gamma,\overline \gamma\right)
= \widetilde \CM\left((\dot \Sigma,\del \dot \Sigma),(M,\vec R);J;\underline \gamma,\overline \gamma\right)/\sim
\ee
to be the set of isomorphism classes of contact instantons $w$.
\end{defn}
This definition does not depend on the choice of $k, \, p$ or $\delta$ as long as $k\geq 2, \, p>2$ and
$\delta > 0$ is sufficiently small. One can also vary $\lambda$ and $J$ and define the universal
moduli space.
(See also \cite{oh:contacton} for the case of
closed strings.)

We consider the space given in \eqref{eq:offshell-space}
$$
\CF\left(\dot \Sigma, \del \dot \Sigma),(M, \vec R);\underline \gamma,\overline \gamma\right)
$$
consisting of smooth maps satisfying the Legendrian boundary condition \
and the asymptotic condition \eqref{eq:limatinfty}.

This being said, since the asymptotic conditions will be fixed, we simply write the space
\eqref{eq:offshell-space} as
$$
\CF((\dot \Sigma,\del \dot \Sigma),(M,\vec R))
$$
and the corresponding moduli space as
$$
\CM((\dot \Sigma,\del \dot \Sigma),(M,\vec R)).
$$
We again consider the covariant linearized operator
$$
D\Upsilon(w): \Omega^0\left(w^*TM,(\del w)^*T\vec R\right) \to
\Omega^{(0,1)}(w^*\Xi) \oplus \Omega^2(\Sigma)
$$
of the section
$$
\Upsilon: w \mapsto \left(\delbar^\pi w, d(w^*\lambda \circ j)\right), \quad
\Upsilon: = (\Upsilon_1,\Upsilon_2)
$$
as before. For the simplicity of notation, we also write
\be\label{eq:CDw}
\CC\CD_{(J,(j,w))}: = \Omega_J^{(0,1)}(w^*\Xi) \oplus \Omega^2(\Sigma)
= \CH^{\pi(0,1)}_w \oplus \Omega^2(\Sigma)
\ee
and
$$
\CC\CD = \bigcup_{(J,(j,w)) \in \CF} \{w\} \times \CC\CD_{(J,(j,w))}.
$$
(Here $\CC\CD$ stands for `codomain'.)

The formula was derived in \cite{oh:contacton}.

\begin{thm}[Theorem 10.1 \cite{oh:contacton}; Proposition 2.21 \cite{oh:contacton-transversality}] 
\label{thm:linearization} In terms of the decomposition
$d\pi = d^\pi w + w^*\lambda\, R_\lambda$
and $Y = Y^\pi + \lambda(Y) R_\lambda$, we have
\bea
D\Upsilon_1(w)(Y) & = & \delbar^{\nabla^\pi}Y^\pi + B^{(0,1)}(Y^\pi) +  T^{\pi,(0,1)}_{dw}(Y^\pi)\nonumber\\
&{}& \quad + \frac{1}{2}\lambda(Y) (\CL_{R_\lambda}J)J(\del^\pi w)
\label{eq:DUpsilon1}\\
D\Upsilon_2(w)(Y) & = &  - \Delta (\lambda(Y))\, dA + d((Y^\pi \rfloor d\lambda) \circ j)
\label{eq:DUpsilon2}
\eea
where $B^{(0,1)}$ and $T_{dw}^{\pi,(0,1)}$ are the $(0,1)$-components of $B$ and
$T_{dw}^\pi$, where $B, \, T_{dw}^\pi: \Omega^0(w^*TM) \to \Omega^1(w^*\Xi)$ are
 zero-order differential operators given by
\be\label{eq:B}
B(Y) =
- \frac{1}{2}  w^*\lambda \otimes \left((\CL_{R_\lambda}J)J Y\right)
\ee
and
\be\label{eq:torsion-dw}
T_{dw}^\pi(Y) = \pi T(Y,dw)
\ee
respectively.
\end{thm}

From the above expression of the covariant linearization of
of the section $\Upsilon = (\Upsilon_1,\Upsilon_2)$, the linearization
continuously extends to a bounded linear map
\be\label{eq:dUpsilon}
D\Upsilon_{(\lambda,T)}(w): T\CW^{k,p} \to \CH^{\pi(0,1)}_{k-1,p}(M,\lambda)
\ee
where we recall
\beastar
 T\CW^{k,p} & = &
\Omega^0_{k,p}\left((w^*TM,(\del w)^*T\vec R);J;\underline \gamma,\overline \gamma\right) \\
\CH^{\pi(0,1)}_{k-1,p}(M,\lambda) & = & \Omega^{(0,1)}_{k-1,p}(w^*\Xi) \oplus \Omega^2_{k-2,p}(\Sigma)
\eeastar
for any choice of $k \geq 2, \, p > 2$. Using the decomposition
$$
\Omega^0_{k,p}\left(w^*TM,(\del w)^*T\vec R\right) \cong \Omega^0_{k,p}\left(w^*\Xi,(\del w)^*T\vec R\right) \oplus
\Omega^0_{k,p}(\dot \Sigma,\del \dot \Sigma)\cdot R_\lambda,
$$
$D\Upsilon(w)$ can be written into the matrix form
\be\label{eq:matrixDUpsilon}
\left(\begin{matrix}\delbar^{\nabla^\pi} + T_{dw}^{\pi,(0,1)} + B^{(0,1)}
 & \frac{1}{2} \lambda(\cdot) (\CL_{R_\lambda}J)J \del^\pi w \\
d\left((\cdot) \rfloor d\lambda) \circ j\right) & -\Delta(\lambda(\cdot)) \,dA
\end{matrix}
\right).
\ee
From the above, we have its parameterized version
\be\label{eq:formula}
D_{J,(j,w)}\aleph_0(L,(0,\xi)) = D\Upsilon_1(w)(Y) + \frac{1}{2}L \circ d^\pi w \circ j 
\ee 
with respect to the contact triad connection associated to the triad $(M,\lambda,J)$.
Here we denote
$$
T^{(1,0)}_{(j,J)}(d^\pi w,\xi) = \frac{1}{2}\left(T(d^\pi w,\xi) + J T(d^\pi w \circ j, \xi)\right).
$$
However if $ w \in
\Upsilon^{-1}\left(o_{\CH^{\pi\prime\prime}} \times o_{ \Lambda^{(1,0)}_{(j,J)}(\cdot)^*\Xi}\right)$, we have 
$d^\pi w(z_0) = 0$ and hence
$$
T^{(1,0)}_{(j,J)}(d^\pi w(z_0), \xi(z_0))=0 =\frac{1}{2}L (w(z_0))
\circ d^\pi w(z_0) \circ j_{z_0}
$$
for any $\xi$.

It follows that the map $D\Upsilon(w)$ is a partial differential operator whose principle
symbol map is given by $\sigma(D\Upsilon) = \sigma(D\Upsilon_1) \oplus \sigma(D\Upsilon_2)$
where
\bea\label{eq:symbol}
\sigma(D\Upsilon_1(w))(\eta) & = & J\Pi^*\eta \nonumber\\
\sigma(D\Upsilon_2(w))(\eta) & = & \langle \lambda,\eta\rangle^2 = (\eta(R_\lambda))^2
\eea
where $\eta$ is a cotangent vector in $T^*M \setminus \{0\}$ and
has decomposition
$$
\eta = \eta^\pi + \eta(R_\lambda(\pi(\eta))\,\lambda(\pi(\eta)).
$$
(See \cite{lockhart-mcowen} for the discussion of general elliptic operators of mixed degree
on noncompact manifolds with cylindrical ends.)

In particular we note that the restriction $D\Upsilon_1(w)|_{\Omega^0(w^*\Xi)}$ has the same
principle symbol as that of
$$
\delbar^{\nabla^\pi} : \Omega^0(w^*\Xi, (\del w)^*\xi) \to \Omega^{(0,1)}(w^*\Xi)
$$
which is the first order elliptic operator of Cauchy-Riemann type, and that
$D\Upsilon_2(w)$ has the symbol of the Hodge Laplacian acting on zero forms
$$
*\Delta: \Omega^0(\dot \Sigma,\del \dot \Sigma) \to \Omega^2(\dot \Sigma).
$$

\subsection{Inclusion of marked points and evaluation maps}

We will treat the two cases,  evaluation at an interior marked point and one at a boundary marked point,
separately. We denote by the subindex $(\ell,k)$ the number of interior and boundary marked points respectively.

Consider the universal moduli space
\beastar
&{}& \CM_{(0,1)}\left((\dot \Sigma, \del \dot \Sigma), (M,R)\right) \\
& = & \left\{((j,w),J, z) \, \Big|\,  w: \Sigma \to
M, \, \Upsilon(J,(j,w))  = 0,\, \, w(\del \dot \Sigma) \subset \vec R,\,  z \in \Int \dot \Sigma \right \}.
\eeastar
The evaluation map $\ev^+: \CM_{(0,1)}\left((\dot \Sigma,\del \dot \Sigma,(M, \vec R)\right) \to M$ is defined by
$$
\ev^+((j,w),z) = w(z).
$$
We then have the fibration
$$
\widetilde \CM_{(0,1)}\left((\dot \Sigma,\del \dot \Sigma),(M, \vec R)\right)  = \bigcup_{J \in \CJ_\lambda} \{J\} \times
\widetilde \CM_{(0,1)}\left ((\dot \Sigma,\del \dot \Sigma),(M, \vec R);J\right)
$$
over $\CJ_\lambda$.
We have the universal ($0$-jet) evaluation map
$$
\Ev^+: \widetilde \CM_{(0,1)}\left((\dot \Sigma,\del \dot \Sigma),(M, \vec R)\right) \to M.
$$
The basic generic transversality is the following from \cite{oh:contacton-transversality}.

\begin{thm}[$0$-jet evaluation transversality]\label{thm:0jet-evaluation}
The evaluation map
$$
\Ev^+: \widetilde \CM_{(0,1)}\left((\dot \Sigma,\del \dot \Sigma),(M, \vec R)\right) \to M
$$
is a submersion. The same holds for the boundary evaluation map
$$
\Ev_\del: \widetilde \CM_{(0,1)}\left((\dot \Sigma,\del \dot \Sigma,)(M, \vec R)\right) \to \vec R.
$$
\end{thm}

Because the details of the proof of the main theorem of the 1-jet evaluation transversality 
largely recycle those used in the proof of this theorem in \cite{oh:contacton-transversality},
we will borrow them later in Section \ref{sec:1jet-proof} which
in turn follows the canonical scheme of the generic 1-jet transversality results proved
in \cite{oh-zhu:ajm}, \cite{oh:higher-jet} for the case of pseudoholomorphic curves in symplectic geometry.

An important ingredient in their proofs is the following
structure theorem  of the distributions with point support from
\cite[Section 4.5]{gelfand},  \cite[Theorem 6.25]{rudin} whose proof we refer readers thereto.

\begin{lem}[Distribution with point support]\label{lem:gelfand}
\index{distribution with point support}
Suppose $\psi$ is a distribution on open
subset $\Omega \subset \R^n$ with $\supp \psi = \{p\}$ and of finite
order $N < \infty$. Then $\psi$ has the form
$$
\psi = \sum_{|\alpha| \leq N} D^\alpha \delta_p
$$
where $\delta_p$ is the Dirac-delta function at $p$ and $\alpha =
(\alpha_1,\ldots, \alpha_n)$ is the multi-indices.
\end{lem}

We start with the case of interior marked point and consider the map
\bea\label{eq:aleph1}
\aleph_0 & : & \CJ_\lambda 
\times \widetilde \CF_{(0,1)}(\dot \Sigma,M) \times \CM_{\dot \Sigma} \to \CC\CD
\times M \nonumber \\
&{}& (J,(j,w),z_0)  \mapsto  (\Upsilon(J,(j,w)),w(z_0)).
\eea
Here the subindex $0$ in $\aleph_0$ stands for the `0-jet'.

We denote by $\pi_i$ the projection from
$\CJ_\lambda \times \widetilde \CF_{(0,1)}(\dot \Sigma,M)$ to the $i$-th factor with $i=1, \, 2$,
and let $p \in M$.
Then we introduce
 \beastar
\widetilde \CM_{(0,1)}(\dot \Sigma,M;\{p\}) & = & \aleph_0^{-1}(o_{\CC\CD} \times \{p\} )\\
\widetilde \CM_{(0,1)}(\dot \Sigma, M;\{p\};J)
& = & \widetilde \CM_{(0,1)}(\dot \Sigma, M;\{p\}) \cap \pi_1^{-1}(J).
\eeastar

The following is a fundamental proposition for the proof of  \cite[Theorem 1.8]{oh:contacton-transversality}
as in the standard strategy exercised in the similar transversality result for the study of
pseudoholomorphic curves in  \cite[Section 10.5]{oh:book1} which in turn follows the scheme used in
\cite{oh-zhu:ajm}  for the 1-jet transversality proof for the case of pseudoholomorphic curves.

\begin{prop}[Proposition 7.2 \cite{oh:contacton-transversality}]\label{prop:0-jet}
The map $\aleph_0$ is transverse to the submanifold
$$
o_{\CC\CD} \times \{p\} \subset \CC\CD \times M.
$$
\end{prop}
We again apply the same scheme with the replacement of pseudoholomorphic curves 
by contact instantons for the 1-jet case in the present paper.
Because the nature of equation is different, especially \emph{because the contact instanton
equation involves the second derivatives}, the proof involves additional complication beyond
that of \cite{oh-zhu:ajm}.

Since the off-shell Fredholm setting used for the proof of this proposition will 
largely reappear  for the 1-jet transversality, we recall the relevant functions spaces
and briefly indicate the main part of the proof now.

Firstly, the linearization $D\aleph_0(J,(j,w),z)$ is given by the map
\be\label{eq:DUpsilon} (L,(b,Y),v) \mapsto
\left(D_{J,(j,w)}\Upsilon(L,(b,Y)), Y(w(z)) + dw(z)(v)\right) \ee
for
$$
L  \in T_J\CJ_\lambda, \, b \in T_j\CM_{\dot \Sigma}, \, v \in
T_z \dot \Sigma , \, Y \in T_w\FF(\Sigma, M;\beta).
$$
This defines a linear map
$$
T_J\CJ_\lambda \times T_w\FF(\Sigma, M;\beta)
\times T_z \dot \Sigma \times T_j\CM_{\dot \Sigma}  \to \CC\CD_{(J,(j,w))} \times T_{w(z)}M
$$
on $W^{1,p}$. But for the map $\aleph_0$ to be differentiable, we need to choose a suitable
completion of $\FF(\dot \Sigma,M)$.

We take the Sobolev completion in the $W^{k,p}$-norm for at least $k \geq 2$. We take $k = 2$.
We would like to prove that this linear map is a submersion
at every element $(J, j, w, z_0) \in \widetilde\CM_1(\dot \Sigma,M)$ i.e.,
at the pair $(w,z_0)$ satisfying
$$
\Upsilon(J,(j,w)) = 0, \quad w(z_0) = p.
$$
For this purpose, we need to study solvability of the system of
equations
 \be D_{J,(j,w)}\Upsilon(L,(b,Y),v) = (\gamma, \omega), \quad
Y(w(z_0)) + dw(v) = X_0
\ee
for any given $(\gamma, \omega) \in \CC\CD_{(J,(j,w))}$ and $X_0$, i.e.,
\be\label{eq:gamma-omega-X0}
\begin{cases}
\gamma \in \Omega_{(j,J)}^{\pi(0,1)}(w^*\Xi), \quad \omega \in  \Omega^2(\dot \Sigma), \\
X_0 \in T_{w(z_0)}M.
\end{cases}
\ee
For the current study of evaluation transversality,
the domain complex structure $j$ does not play much
role in our study. Especially it does not play any role throughout
our calculations except that it appears as a parameter. Therefore we will fix $j$
throughout the proof. Then it
will be enough to consider the case $b = 0 $. Then the above equation
is reduced to
\be\label{eq:b=0}
D_{J,w}\Upsilon (L,Y) = (\gamma,\omega), \quad Y(w(z_0)) + dw(v) = X_0.
\ee

Firstly,  we study \eqref{eq:b=0} for $Y \in W^{2,p}$. We regard
$$
\CC\CD_{(J,(j,w))} \times T_{w(z_0)}M
$$
as a Banach space with the norm $\|\cdot \|_{1,p} + \|\cdot \|_p +  |\cdot|$,
where $|\cdot|$ is any norm induced by an inner product on $T_xM$.

The following explicit characterization from \cite{oh:contacton-transversality}
of the $L^2$-adjoint of the differential operator
$D\Upsilon(J,(j,w))$ is crucial for the study of $\coker D\Upsilon(J,(j,w))$ and
$\coker D\aleph_0(J,(j,w),z)$. Actually \cite[Proposition 4.4]{oh:contacton-transversality} 
handles more general cases than what we need here, since it considers the more complicated case of
the boundary evaluation transversality. Here we state a somewhat simpler version thereof
considering the interior marked points instead of the boundary ones therefrom.

\begin{prop}[Compare with Proposition 4.4 \cite{oh:contacton-transversality}]
\label{prop:L2-cokernel} 
Consider the pair
$$
(\eta,g) \in \CC\CD_{(J,(j,w))}
$$
and recall the correspondence
$$
(\eta,g) \longleftrightarrow \eta + g\, R_\lambda
$$
given in terms of the decomposition $TM = \Xi \oplus \R\langle R_\lambda \rangle$. 

Suppose that the pair $(\eta,g)$ lies in the $L^2$-cokernel of $D\Upsilon^{\CL eg}(w,R)$.
Then the following hold and vice versa:
\begin{enumerate}
\item
$\eta$ satisfies
\be\label{eq:adjoint-eta}
(\delta^{\nabla^{\pi(0,1)} } + T^{\pi(1,0)} + B^{(1,0)}) (\eta^{(1,0)})
- J  \langle dg, dw \rangle  = 0  \quad \text{\rm on }\,\dot \Sigma.
\ee
\item  The function $g$ satisfies
\be\label{eq:adjoint-g}
\Delta g -  \frac 12 g \left \langle (\CL_{R_\lambda}J)( J d^\pi w), \eta
\right \rangle  = 0,  \quad \text{\rm on }\, \dot \Sigma
\ee
\end{enumerate}
\end{prop}

The equations \eqref{eq:adjoint-eta} and \eqref{eq:adjoint-g}
are nothing but the analog of the equation right above \cite[Equation (3.13)]{oh:fredholm}.
In terms of the isothermal coordinates $(x,y)$ adapted to $\del \dot \Sigma$ and
in the usual convention for the calculus of vector valued differential forms
(see \cite[Appendix B]{oh-wang:CR-map1}), we can
express
\be\label{eq:eta10}
\langle dg, dw \rangle : =  dg \wedge * dw
= \left(\frac{\del g}{\del x} \frac{\del w}{\del x} + \frac{\del g}{\del y} \frac{\del w}{\del y}\right)
\ee
for the one-form $\eta = \eta_x \, dx + \eta_y \, dy$.

Then the following then finishes the proof by the ellipticity
of the linearization map. 

\begin{prop}\label{prop:0jet-dense}
The subspace
$$
\Image \aleph_0 \subset \CB_0 \oplus T_{w(z_0)}M
$$
is dense.
\end{prop}

\begin{notation}
For the clarification of notations, we denote the natural pairing
$$
\CB \times \CB^* \to \R
$$
by $\langle \cdot, \cdot \rangle$ and the inner product on $T_xM$ by $(\cdot,
\cdot)_{x}$.
\end{notation}

The following characterization of the critical point is now obvious to
see, which however is a key ingredient for the Fredholm framework 
used for the off-shell setting of our proof of evaluation transversality against
the contact distribution in \cite{oh:contacton-transversality}.

We recall the splitting $TM = \Xi \oplus \R \langle R_\lambda \rangle$.
\begin{lem}
For any $((j,w),z) \in \widetilde \CM_{(0,1)}(\dot \Sigma, M)$, since
$\delbar^\pi_{J,j}w=0$, we have \be \label{eq:key} w^*\lambda(z) = 0 \quad
\mbox{if and only if} \quad \del^\pi_{(j,J)}w(z) \in \Lambda^1(\Xi_{w(z)}).
\ee
\end{lem}

\section{Fredholm theory of the linearized operator}
\label{subsec:fredholm}

Let $(\dot \Sigma, j)$ be a punctured Riemann surface, the set of whose punctures
may be empty, i.e., $\dot \Sigma = \Sigma$ is either a closed or a punctured
Riemann surface. In this subsection and the next, we lay out the precise relevant off-shell framework
of functional analysis, and  establish the Fredholm property of the linearization map.

\subsection{Fredholm theory on punctured bordered Riemann surfaces}

By the (local) ellipticity shown in the previous subsection, it remains to examine the
Fredholm property of the linearized operator $D\Upsilon(w)$ on the punctured Riemann
surfaces. For this purpose, we need to examine the asymptotic behavior of the operator 
near punctures in strip-like coordinates.

We first decompose the section $Y \in w^*TM$ into
$$
Y = Y^\pi + \lambda(Y) R_\lambda.
$$
Then we have
\bea
D\Upsilon_1^1(w)(Y^\pi) & = & \delbar^{\nabla^\pi}Y^\pi + B^{(0,1)}(Y^\pi) 
+  T^{\pi,(0,1)}_{dw}(Y^\pi), \label{eq:D}\\
D\Upsilon_2^1(w)(Y^\pi) & = & d((Y^\pi \rfloor d\lambda) \circ j),\label{eq:DUpsilon21}\\
D\Upsilon_1^2(w)(\lambda(Y) R_\lambda) & = & \frac12 \lambda(Y) \CL_{R_\lambda}J J \del^\pi w,
\label{eq:DUpsilon12}\\
D\Upsilon_2^2(w)(\lambda(Y) R_\lambda) & = & - \Delta(\lambda(Y))\, dA. \label{eq:DUpsilon22}
\eea
Noting that $Y^\pi$ and $\lambda(Y)$ are independent of each other, we write
$$
Y = Y^\pi + f R_\lambda, \quad f: = \lambda(Y)
$$
where $f: \dot \Sigma \to \R$ is an arbitrary function satisfying the boundary condition
$$
Y^\pi(\del \dot \Sigma) \subset T\vec R, \quad f(\del \dot \Sigma) = 0
$$
by the Legendrian boundary condition satisfied by $Y$. The following is obvious from
the expression of the $D\Upsilon_i^j(w)$.
\begin{lem}[Lemma 3.17 \cite{oh:perturbed-contacton}]
 Suppose that $w$ is a solution to \eqref{eq:contacton-Legendrian-bdy-intro}.
The operators $D\Upsilon_i^j(w)$ have the following continuous extensions:
\beastar
D\Upsilon_1^1(w)(Y^\pi)& : & \Omega^0_{k,p}(w^*\Xi,(\del w)^*T\vec R) \to \Omega^{(0,1)}_{k-1,p}(w^*\Xi) \\
D\Upsilon_2^1(w)(Y^\pi) &: & \Omega^0_{k,p}(w^*\Xi,(\del w)^*T\vec R)
 \to
\Omega^2_{k-1,p}(\dot \Sigma) \hookrightarrow \Omega^2_{k-2,p}(\dot \Sigma) \\
D\Upsilon_1^2(w)((\cdot) R_\lambda) & : & \Omega^0_{k,p}(\dot \Sigma,\del \dot \Sigma)
 \to
\Omega^2_{k,p}(\dot \Sigma) \hookrightarrow \Omega^2_{k-2,p}(\dot \Sigma) \\
D\Upsilon_2^2(w)((\cdot) R_\lambda) & : & \Omega^0_{k,p}(\dot \Sigma,\del \dot \Sigma) \to \Omega^2_{k-2,p}(\Sigma).
\eeastar
\end{lem}

We regard the domains of  $D\Upsilon_i^2$ for $i=1, \,2$
as $C^\infty(\dot \Sigma, \del \dot \Sigma)$ using the isomorphism
$$
C^\infty(\dot \Sigma, \del \dot \Sigma) \cong \Omega^0(\dot \Sigma, \del \dot \Sigma) \otimes R_\lambda.
$$
The following Fredholm property of the linearized operator is established \cite{oh:contacton-transversality}
and its index is computed in \cite{oh-yso:index}.

\begin{prop}\label{prop:closed-fredholm}
Suppose that $w$ is a solution to \eqref{eq:contacton-Legendrian-bdy-intro}.
Consider the completion of $D\Upsilon(w)$,
which we still denote by $D\Upsilon(w)$, as a bounded linear map
from $\Omega^0_{k,p}(w^*TM,(\del w)^*T\vec R)$ to
$\Omega^{(0,1)}(w^*\Xi)\oplus \Omega^2(\Sigma)$
for $k \geq 2$ and $p \geq 2$. Then
\begin{enumerate}
\item The off-diagonal terms of $D\Upsilon(w)$ are relatively compact operators
against the diagonal operator.
\item
The operator $D\Upsilon(w)$ is homotopic to the operator
\be\label{eq:diagonal}
\left(\begin{matrix}\delbar^{\nabla^\pi} + T_{dw}^{\pi,(0,1)}+ B^{(0,1)} & 0 \\
0 & -\Delta(\lambda(\cdot)) \,dA
\end{matrix}
\right)
\ee
via the homotopy
\be\label{eq:s-homotopy}
s \in [0,1] \mapsto \left(\begin{matrix}\delbar^{\nabla^\pi} + T_{dw}^{\pi,(0,1)} + B^{(0,1)}
& \frac{s}{2} \lambda(\cdot) (\CL_{R_\lambda}J)J (\pi dw)^{(1,0)} \\
s\, d\left((\cdot) \rfloor d\lambda) \circ j\right) & -\Delta(\lambda(\cdot)) \,dA
\end{matrix}
\right) =: L_s
\ee
which is a continuous family of Fredholm operators.
\item And the principal symbol
$$
\sigma(z,\eta): w^*TM|_z \to w^*\Xi|_z \oplus \Lambda^2(T_z \dot \Sigma ), \quad 0 \neq \eta \in T^*_z\Sigma
$$
of \eqref{eq:diagonal} is given by the matrix
\beastar
\left(\begin{matrix} \frac{\eta + i\eta \circ j}{2} Id  & 0 \\
0 & |\eta|^2
\end{matrix}\right).
\eeastar
\end{enumerate}
\end{prop}
\begin{proof} Statement (1) is a consequence of compactness of Sobolev embedding
$$
\Omega^2_{k-1,p}(\dot \Sigma) \hookrightarrow \Omega^2_{k-2,p}(\dot \Sigma), \quad
\Omega^2_{k,p}(\dot \Sigma) \hookrightarrow \Omega^2_{k-2,p}(\dot \Sigma).
$$
When $\del \dot \Sigma = \emptyset$, the same kind of statement is proved in \cite{oh:contacton}.
Essentially the same proof applies by incorporating the boundary condition.
\end{proof}

With these preparations, the following is a corollary of exponential estimates established
in \cite[Part II]{oh-wang:CR-map2}, \cite{oh:contacton-Legendrian-bdy} for bordered
contact instantons with Legendrian boundary condition.

\begin{prop}[Section 7 \cite{oh:contacton-Legendrian-bdy}]
Let
$$
\underline \gamma = \{\gamma^+_1, \ldots, \gamma^+_{s^+}\}, \quad
\overline \gamma = \{\gamma^-_1, \ldots, \gamma^-_{s^-}\}
$$
be the given $\lambda$-Reeb chords of  $\vec R = (R_1, \ldots, R_{s^+ + s^-}) $ which
are nondegenerate. If $w$ has finite energy,  then we have
$$
w \in \CW^{k,p}\left((\dot \Sigma,\del \dot \Sigma),(M, \vec R);J;\underline \gamma,\overline \gamma\right).
$$
\end{prop}

Now we are ready to wrap-up the discussion of the Fredholm property of
the linearization map
$$
D\Upsilon_{(\lambda,T)}(w): \Omega^0_{k,p}\left(w^*TM, (\del w)^*T\vec R;J;\underline \gamma,\overline \gamma\right) \to
\Omega^{(0,1)}_{k-1,p}(w^*\Xi) \oplus \Omega^2_{k-2,p}(\dot \Sigma)
$$
by proving Statement (1) of Proposition \ref{prop:closed-fredholm}.

The following proposition can be derived from the arguments used by
Lockhart and McOwen \cite{lockhart-mcowen}.

\begin{prop}\label{prop:fredholm} Assume that $\underline \gamma, \, \overline \gamma$ are
nondegenerate. Then the operator
\eqref{eq:matrixDUpsilon} is Fredholm.
\end{prop}
By  the continuous invariance of the Fredholm index, we obtain
\be\label{eq:indexDXiw}
\operatorname{Index} D\Upsilon_{(\lambda,T)}(w) =
\operatorname{Index} \left(\delbar^{\nabla^\pi} + T^{\pi,(0,1)}_{dw}  + B^{(0,1)}\right) + \operatorname{Index}(-\Delta).
\ee
The computation of index is given in \cite{oh:contacton} for the closed string case
and in \cite{oh-yso:index} for the open string case.

\subsection{Generic mapping transversality of (relative) contact instantons}
\label{sec:moduli-space}

In this section, we study the mapping transversality under the perturbation of
$J$.
Now we involve the set $\CJ(M,\lambda)$ of $\lambda$-adaptable almost
complex structures which we recall here.
\begin{defn}[Contact triad \cite{oh-wang:connection}]
\label{defn:adapted-J} Let $(M,\xi)$ be a contact manifold, and $\lambda$ be a contact form of $\xi$.
An endomorphism $J: TM \to TM$ is called a \emph{$\lambda$-adapted CR-almost complex structure}
if it satisfies
\begin{enumerate}
\item $J(\xi) \subset \xi$, $JR_\lambda = 0$ and $J|_\xi^2 = -id|_\xi$,
\item $g_\xi: = d\lambda(\cdot, J|_{\xi} \cdot)|_{\xi}$ defines a Hermitian vector bundle $(\xi,J|_{\xi},g_\xi)$.
\end{enumerate}
We call the triple $(M,\lambda, J)$ a \emph{contact triad}.
\end{defn}

We study the linearization of the map $\Upsilon^{\text{\rm univ}}$
which is the map $\Upsilon$ augmented by
the argument $J \in \CJ(M,\lambda)$. More precisely, we define the universal section
$$
\Upsilon^{\text{\rm univ}}: \CM(\dot \Sigma) \times \CF \times  \CJ(M,\lambda) \to \CH^{\pi(0,1)}(M,\lambda)
$$
given by
\be\label{eq:Upsilon-univ}
\Upsilon^{\text{\rm univ}}(j, w, J) = \left(\delbar_J^\pi w, d(w^*\lambda \circ j)\right)
=: \Upsilon_J(w,j)
\ee
and study its linearization at each $(j,w,J) \in (\Upsilon^{\text{\rm univ}})^{-1}(0)$.
In the discussion below, we will fix the complex
structure $j$ on $\Sigma$, and so suppress $j$ from the argument of $\Upsilon^{\text{\rm univ}}$.

The following universal linearization formula plays a crucial role in the generic transversality
result as in the case of pseudoholomorphic curves in symplectic geometry.

\begin{lem}\label{lem:DY-univ} Denote by $L: = \delta J$ the first variation of $J$.
We have the linearization
$$
D_{(w,J)} \Upsilon^{\text{\rm univ}}: T_j \CM_{\dot \Sigma} \times
T_w \CF \times T_J \CJ(M,\lambda) \to \Omega^{(0,1)}(w^*\Xi) \times
 \Omega^2(\dot \Sigma)\otimes R_\lambda
$$
whose explicit formula is given by
$$
D_{(w,J)} \Upsilon^{\text{\rm univ}}(Y,L) = D_1 \Upsilon^{\text{\rm univ}}(Y)  + D_2 \Upsilon^{\text{\rm univ}}(L)
$$
where we have partial derivatives
\be\label{eq:D2}
D_1 \Upsilon^{\text{\rm univ}} (Y) = D\Upsilon(Y), \quad D_2 \Upsilon^{\text{\rm univ}}(L)
= \frac12 L( d^\pi w \circ j)
\ee
\end{lem}
\begin{proof} This is straightforward from the definition
$$
\delbar^\pi w = \frac{d^\pi w + J d^\pi w \circ j}{2}
$$
and the fact that the projection $\pi$ does not depend on the choice of $J$ but depends
only on $\lambda$. We omit its proof.
\end{proof}

Following the procedure of considering the set $\CJ^\ell(M,\lambda)$ of $\lambda$-adapted
$C^\ell$ CR-almost complex structures $J$ inductively as $\ell$ grows (see \cite{mcduff-salamon-symplectic}, \cite[Section 10.4]{oh:book1}
for the detailed explanation),
we denote the zero set $(\Upsilon^{\text{\rm univ}})^{-1}(0)$ by
$$
\CM(M,\lambda,\vec R;\overline \gamma, \underline \gamma) = \left\{ (w,J)
\in \CW^{k,p}(\dot \Sigma, M;\overline \gamma, \underline \gamma)) \times \CJ^\ell(M,\lambda)
\, \Big|\, \Upsilon^{\text{\rm univ}}(w, J) = 0 \right\}
$$
which we call the universal moduli space. Denote by
$$
\Pi_2: \CW^{k,p}(\dot \Sigma, M, \vec R;\overline \gamma, \underline \gamma) \times \CJ^\ell(M,\lambda) \to
\CJ^\ell(M,\lambda)
$$
the projection. Then we have
\be\label{eq:MMK}
\CM(J;\overline \gamma, \underline \gamma)= \CM(M,\lambda, \vec R; J;\overline \gamma, \underline \gamma)
 = \Pi_2^{-1}(J) \cap \CM(M,\lambda,  \vec R;\overline \gamma, \underline \gamma).
\ee

We state the following standard statement that often occurs in this kind of generic
transversality statement via the Sard-Smale theorem.

\begin{thm}\label{thm:trans} Let $0 < \ell < k -\frac{2}{p}$.
Consider the moduli space $\CM(M,\lambda;\overline \gamma, \underline \gamma)$. Then
\begin{enumerate}
\item $\CM(M,\lambda, \vec R;\overline \gamma, \underline \gamma)$ is
an infinite dimensional $C^\ell$ Banach manifold.
\item The projection
$$
\Pi_2|_{(\Upsilon^{\text{\rm univ}})^{-1}(0)} : (\Upsilon^{\text{\rm univ}})^{-1}(0) \to \CJ^\ell(M,\lambda)
$$
is a Fredholm map and its index is the same as that of $D\Upsilon(w)$
for a (and so any) $w \in  \CM(M,\lambda, \vec R; J;\overline \gamma, \underline \gamma)$.
\end{enumerate}
\end{thm}

An immediate corollary of Sard-Smale theorem is that for a generic choice of $J$
$$
\Pi_2^{-1}(J) \cap (\Upsilon^{\text{\rm univ}})^{-1}(0)= \CM(J;\overline \gamma, \underline \gamma)
$$
is a smooth manifold: One essential ingredient for the generic transversality under the perturbation of
$J \in \CJ(M,\lambda)$ is the usage of the following unique continuation result.

\begin{prop}[Unique continuation lemma; Proposition 12.3 \cite{oh:contacton}]
\label{prop:unique-conti}
Any non-constant contact Cauchy-Riemann map does not
have an accumulation point in the zero set of $dw$.
\end{prop}

A generic transversality under the perturbation of boundaries has been also established in
\cite{oh:contacton-transversality}.

\begin{thm}[Theorem 4.4 \cite{oh:contacton-transversality}]
\label{thm:trans-under-bdy} Let $(M,\xi)$ be a contact manifold, and let
$\lambda$ a contact form  be given. We consider the same equation
considered in Theorem \ref{thm:trans}. Fix $J$ and $k$ and consider
$\vec R$ that intersect transversally pairwise and no triple intersections.

Then there exists a residual subset of $\vec R = (R_0, \ldots, R_k)$
of Legendrian submanifolds such that  the moduli space
$\CM(M,\lambda, \vec R;\overline \gamma, \underline \gamma) $ is transversal.
\end{thm}

\section{The (0-jet) evaluation transversality: Statement}
\label{sec:ev-transverse-statement}

We first recall the off-shell setting of the study of linearized operator in Theorem \ref{thm:linearization}.

We consider the moduli space
$$
\CM_k((\dot \Sigma,\del \dot \Sigma),(M,\vec R)), \quad \vec R = (R_1,\cdots, R_k)
$$
 of finite energy maps $w: \dot \Sigma \to M$ satisfying the equation
\eqref{eq:contacton-Legendrian-bdy-intro} as before.

We consider the space given in \eqref{eq:offshell-space}
$$
\CF(\dot \Sigma, \del \dot \Sigma),(M, \vec R);\underline \gamma,\overline \gamma)
$$
consisting of smooth maps satisfying the boundary condition \eqref{eq:bdy-condition}
and the asymptotic condition \eqref{eq:limatinfty}.

This being said, since the asymptotic conditions will be fixed, we simply write the space
\eqref{eq:offshell-space} as
$$
\CF((\dot \Sigma,\del \dot \Sigma),(M,\vec R))
$$
and the corresponding moduli space as
$$
\CM((\dot \Sigma,\del \dot \Sigma),(M,\vec R)).
$$
We again consider the covariant linearized operator
$$
D\Upsilon(w): \Omega^0(w^*TM,(\del w)^*T\vec R) \to
\Omega^{(0,1)}(w^*\Xi) \oplus \Omega^2(\Sigma)
$$
of the section
$$
\Upsilon: w \mapsto \left(\delbar^\pi w, d(w^*\lambda \circ j)\right), \quad
\Upsilon: = (\Upsilon_1,\Upsilon_2)
$$
as before. 

We will treat the two cases,  evaluation at an interior marked point and one at a boundary marked point,
separately. We denote by the subindex $(k, \ell)$ the number of interior and boundary marked points respectively.

Consider the universal moduli space
\beastar
&{}& \CM_{(0,1)}((\dot \Sigma, \del \dot \Sigma, (M,R)) \\
& = & \{((j,w),J, z) \mid w: \Sigma \to
M, \, \Upsilon(J,(j,w))  = 0,\, \, w(\del \dot \Sigma) \subset \vec R,\,  z \in \Int \dot \Sigma \}.
\eeastar
The evaluation map $\ev^+: \CM_{(0,1)}((\dot \Sigma,\del \dot \Sigma,(M, \vec R)) \to M$ is defined by
$$
\ev^+((j,w),z) = w(z)
).
$$
We then have the fibration
$$
\widetilde \CM_{(0,1)}\left((\dot \Sigma,\del \dot \Sigma),(M, \vec R)\right)  = \bigcup_{J \in \CJ_\lambda}
\widetilde \CM_{(0,1)}\left((\dot \Sigma,\del \dot \Sigma),(M, \vec R);J\right)
\to \CJ_\lambda.
$$
We have the universal ($0$-jet) evaluation map
$$
\Ev^+: \widetilde \CM_{(0,1)}\left((\dot \Sigma,\del \dot \Sigma),(M, \vec R)\right) \to M.
$$
The basic generic transversality is the following.

\begin{thm}[$0$-jet evaluation transversality]\label{thm:0jet}
The evaluation map
$$
\Ev^+: \widetilde \CM_{(0,1)}\left((\dot \Sigma,\del \dot \Sigma),(M, \vec R)\right) \to M
$$
is a submersion. The same holds for the boundary evaluation map
$$
\Ev_\del: \widetilde \CM_{(0,1)}\left((\dot \Sigma,\del \dot \Sigma,)(M, \vec R)\right) \to \vec R.
$$
\end{thm}

\section{The (0-jet) evaluation transversality: Proof}
\label{sec:ev-transverse-proof}

We closely follow the scheme exercised for the proof of
evaluation transversality given in \cite[Section 10.5]{oh:book1}  closely which
in turn follows the scheme of the generic 1-jet transversality results proved
in \cite{oh-zhu:ajm}, \cite{oh:higher-jet} for the case of pseudoholomorphic curves in symplectic geometry.

We start with the case of interior marked point and consider the map
\bea\label{eq:aleph1}
\aleph_0 & : & \CJ_\lambda \times \CM_{\dot \Sigma}
\times \widetilde \CF_{(1,0)}(\dot \Sigma,M) \to \CC\CD
\times M \nonumber \\
&{}& (J,(j,w),z_0)  \mapsto  (\Upsilon(J,(j,w)),w(z_0)).
\eea
Here the subindex $0$ in $\aleph_0$ stands for the `0-jet'.

We denote by $\pi_i$ the projection from
$\CJ_\lambda \times \widetilde \CF_{(1,0)}(\dot \Sigma,M)$ to the $i$-th factor with $i=1, \, 2$.
Then we introduce
 \beastar
\widetilde \CM_{(0,1)}(\dot \Sigma,M;\{p\}) & = & \aleph_0^{-1}(o_{\CC\CD} \times \{p\} )\\
\widetilde \CM_{(0,1)}(\dot \Sigma, M;\{p\};J)
& = & \widetilde \CM_{(0,1)}(\dot \Sigma, M;\{p\}) \cap \pi_1^{-1}(J).
\eeastar

The following is a fundamental proposition for the proof of Theorem \ref{thm:0-jet}
as in the standard strategy exercised in the similar transversality result for the study of
pseudoholomorphic curves in  \cite[Section 10.5]{oh:book1} which in turn follows the scheme used in
\cite{oh-zhu:ajm}  for the 1-jet transversality proof for the case of pseudoholomorphic curves.
We apply the same scheme with the replacement of pseudoholomorphic curves by contact instantons
first for the 0-jet case in this part and for the 1-jet case in the next part.
Because the nature of equation is different, especially \emph{because the contact instanton
equation involves the second derivatives}, the proof involves additional complication beyond
that of \cite{oh-zhu:ajm}.

\begin{thm}\label{thm:0-jet}
The map $\aleph_0$ is transverse to the submanifold
$$
o_{\CC\CD} \times \{p\} \subset \CC\CD \times M.
$$
\end{thm}
\begin{proof} Its linearization $D\aleph_0(J,(j,w),z)$ is given by the map
\be\label{eq:DUpsilon} (L,(b,Y),v) \mapsto
\left(D_{J,(j,w)}\Upsilon(L,(b,Y)), Y(w(z)) + dw(z)(v)\right) \ee
for
$$
L  \in T_J\CJ_\lambda, \, b \in T_j\CM_{\dot \Sigma}, \, v \in
T_z \dot \Sigma , \, Y \in T_w\FF(\Sigma, M;\beta).
$$
This defines a linear map
$$
T_J\CJ_\lambda \times T_j\CM_{\dot \Sigma} \times T_w\FF(\Sigma, M;\beta)
\times T_z \dot \Sigma  \to \CC\CD_{(J,(j,w))} \times T_{w(z)}M
$$
on $W^{1,p}$. But for the map $\aleph_0$ to be differentiable, we need to choose the
completion of $\FF(\dot \Sigma,M)$.

We take the Sobolev completion in the $W^{k,p}$-norm for at least $k \geq 2$. We take $k = 2$.
We would like to prove that this linear map is a submersion
at every element $(J, j, w, z_0) \in \widetilde\CM_1(\dot \Sigma,M)$ i.e.,
at the pair $(w,z_0)$ satisfying
$$
\Upsilon(J,(j,w)) = 0, \quad w(z_0) = p.
$$
For this purpose, we need to study solvability of the system of
equations
 \be D_{J,(j,w)}\Upsilon(L,(b,Y),v) = (\gamma, \omega), \quad
Y(w(z_0)) + dw(v) = X_0
\ee
for any given $(\gamma, \omega) \in \CC\CD_w$ and $X_0$, i.e.,
$$
\gamma \in \Omega_{(j,J)}^{\pi(0,1)}(w^*TM), \,
\omega \in \Omega^2(\dot \Sigma), \quad X_0 \in T_{w(z_0)}M.
$$
For the current study of evaluation transversality,
the domain complex structure $j$ does not play much
role in our study. Especially it does not play any role throughout
our calculations except that it appears as a parameter. Therefore we will fix $j$
throughout the proof. Then it
will be enough to consider the case $b = 0 $. Then the above equation
is reduced to
\be\label{eq:b=0}
D_{J,w}\Upsilon (L,Y) = (\gamma,\omega), \quad Y(w(z_0)) + dw(v) = X_0.
\ee

Firstly,  we study \eqref{eq:b=0} for $Y \in W^{2,p}$. We regard
$$
\CC\CD_{(J,(j,w))} \times T_{w(z_0)}M
$$
as a Banach space with the norm $\|\cdot \|_{1,p} + \|\cdot \|_p +  |\cdot|$,
where $|\cdot|$ is any norm induced by an inner product on $T_xM$.

We will show that the image of the map
\eqref{eq:DUpsilon} restricted to the elements of the form
$$
(L,(0,Y),v)
$$
is onto as a map
$$
T_J \CJ_\lambda \times \Omega^0_{2,p}(w^*TM)
\to \CC\CD_{J,w}^{1,p} \times T_{w(z_0)}M
$$
where $(w,j,z_0,J)$ lies in $\Upsilon_1^{-1}(o_{\CH^{(0,1)}}\times \Xi)$, and we set
\be\label{eq:CD1p}
\CC\CD_{J,w}^{1,p} := \Omega^{(0,1)}_{1,p}(w^*\Xi) \times \Omega^2_p (\dot \Sigma).
\ee
For the clarification of notations, we denote the natural pairing
$$
\CB \times \CB^* \to \R
$$
by $\langle \cdot, \cdot \rangle$ and the inner product on $T_xM$ by $(\cdot,
\cdot)_{x}$.

We now prove the following which will then finish the proof by the ellipticity
of the linearization map. The remaining part of proof will be occupied by the proof
of this statement.

\begin{prop}\label{prop:1jet-dense}
The subspace
$$
\Image \aleph_0 \subset \CB_0 \oplus T_{w(z_0)}M
$$
is dense.
\end{prop}
\begin{proof}
Let $((\eta, f)X_p) \in \CB^* \times T_pM$ satisfy
\be\label{eq:0jet=0}
\left\langle D_w\Upsilon(J,(j,w))Y + \left(\frac{1}{2}L \cdot d^\pi w \circ j,0\right),(\eta,f) \right\rangle
+ \langle Y, \delta_{z_0} X_p \rangle = 0
\ee
for all $Y \in \Omega^0_{2,p}(w^*TM)$ and $L$ where $\delta_{z_0}$ is the
Dirac-delta function supported at $z_0$.
By the Hahn-Banach lemma, it will be
enough to prove
\be\label{eq:0jet-vanishing}
(\eta,f) = 0 = X_p.
\ee
In the derivation of \eqref{eq:0jet=0},  we have used the formula
$$
\Upsilon(J,(j,w)) = \left(\delbar^\pi_J w, d(w^*\circ \lambda \circ j)\right)
$$
to compute the linearization of $\aleph_0$. In particular, we have
$$
D_1\aleph_0(L)) = \left(\frac{1}{2}L \cdot d^\pi w \circ j,0\right)
$$
in the direction of $J$ where the second factor of the value of $\Upsilon$
does not depend on $J$.  Obviously, we have
$D_2\aleph_0(Y)(Y) = D\Upsilon(Y)$.

Under this assumption, we would like to show \eqref{eq:0jet-vanishing}.
Without loss of any generality, we
may assume that $Y$ is smooth as before since $C^\infty(w^*TM)
\hookrightarrow \Omega^0_{2,p}(w^*TM)$ is dense.
Taking $L=0$ in \eqref{eq:0jet=0}, we obtain
\be \label{eq:coker} %
\langle D_w\Upsilon(J,(j,w)) Y, (\eta,f) \rangle + \langle  Y,
\delta_{z_0} X_p \rangle = 0 \quad \mbox{ for all $Y$
of $C^{\infty}$ }.
\ee %
Therefore by definition of the distribution derivatives, $\eta$ satisfies
$$
(D_w\Upsilon(J,(j,w)))^\dagger (\eta,f) - \delta_{z_0} X_p = 0
$$
as a distribution, i.e.,
$$
(D_w\Upsilon(J,(j,w)))^\dagger (\eta,f) = \delta_{z_0} X_p
$$
where $(D_w \Upsilon(J,(j,w)))^\dagger$ is the formal adjoint of $D_w
\Upsilon(J,(j,w))$ whose symbol is of the same type as $D_w \Upsilon_{(j,J)}$ and so is
an elliptic first order differential operator. (See \eqref{eq:matrixDUpsilon} for the linearization formula and
recall that $(\delbar_J^\pi)^\dagger = - \del_J^\pi$ modulo zero order operators.)
By the elliptic regularity, $(\eta,f)$ is a classical solution on $\Sigma \setminus \{z_0\}$.

On the other hand, by setting $Y = 0$ in \eqref{eq:0jet=0},  we get
\be
\label{eq:1/2L} \langle L\cdot dw\circ j, (\eta,f) \rangle=0
\ee
for all $L\in T_{J} \CJ_\lambda$. From this identity, the
argument used in the transversality proven in the previous section
shows that $\eta=0$ in a small neighborhood of any somewhere injective point in $\Sigma
\setminus \{\chi^+\}$. 

We now show that in the case of contact instantons we can bypass this somewhere
injectivity hypothesis as follows.
Recall that we always assume that $\dot \Sigma$ carries at least one puncture for the
moduli space to contain an element whose image is completely contained in a Reeb trajectory.
 By the finite energy condition and by the removable singularity
theorem \cite[Theroem 8.7]{oh:contacton}, the asymptotic limit of $w$ in the
strip-like coordinates is an iso-speed Reeb chord $(\gamma,T)$ 
for $\gamma;[0,1] \to M$ satisfying $\dot \gamma = T R_\lambda(\gamma(t))$ with $T \neq 0$.
By the genericity hypothesis stated in Theorem \ref{thm:Reeb-chords} (3),
the chord cannot be a closed Reeb chord. Therefore we can decompose
$$
[0,1] = I_+ \cup I_-
$$
so that $\text{\rm mult}_\gamma(p)$ is even for $p \in I_+$ and odd for $p \in I_-$.
This shows that  when $w$ is not somewhere injective near puncture,
the multiplicity of $w$ is not constant near the relevant puncture. Then the 
standard argument of the transversality still concludes $\eta = 0$
by the unique continuation theorem, on $ \Sigma \backslash\{z_0\}$ and 
so the support of $\eta$ as a distribution
on $\Sigma$ is contained at the one-point subset $\{z_0\}$ of $\Sigma$.

The following lemma will then conclude the proof of Proposition \ref{prop:0jet-dense}.
We omit its proof since we will give the proof of more nontrivial statement given in
Lemma \ref{lem:eta=0} for the 1-jet evaluation case.

\begin{lem} \label{lem:eta=0}  $\eta$ is a distributional solution of
$(D_w \Upsilon(J,(j,w)))^\dagger (\eta,\omega) = 0$ on $\Sigma$
and so continuous.
In particular, we have $(\eta,\omega) = 0$ in $(\CC\CD)^*$.
\end{lem}

Once we know $(\eta,\omega) = 0$, the equation \eqref{eq:0jet=0} is reduced to
the finite dimensional equation
\be\label{eq:simple0}
(Y(z_0),X_p)_{z_0} = 0
\ee
It remains to show that $X_p = 0$. For this, we have only to
show that the image of the evaluation map
$$
Y \mapsto  Y(z_0)
$$
is surjective onto $T_{p}M$, which is now obvious.
\end{proof}

This in turn concludes the proof of Theorem \ref{thm:0-jet}.
\end{proof}

\section{Generic 1-jet evaluation transversality: Statement}

In this section, we study the 1-jet transversality of the derivative of the evaluation map
$\ev:\widetilde \CM_{(0,1)} (\dot \Sigma,M;J) \to M$.
We will be particularly interested in the transversality result against the submanifold
$$
\Xi \subset TM
$$
for the contact distributions. For the study, the following reformulation of this transversality
when a contact form $\lambda$ is given.

\begin{lem}\label{lem:lambda-transversality}
Let $\lambda$ be a contact form of a contact manifold $(M,\Xi)$.
Consider $\varphi: N \to M$ be any smooth map. Then the differential
$$
d\varphi: TN \to TM
$$
is transversal to $\Xi$ if and only if the function $TN \to \R$
defined by $(x,v) \mapsto \lambda_x(d\varphi(v))$ has 0 as a regular value.
\end{lem}

The following moduli space with 1-jet constraint will play a crucial role in our
proof of Weinstein' conjecture \cite{oh:weinstein-conjecture}.

\begin{defn}[Moduli space with $\Xi$-tangency condition]
We call  the 1-marked moduli space
$$
\CM_{(1,0)}^\Xi(\dot \Sigma, M;J)
=\{(w,z^+) \mid w \in \CM_{(0,1);[0,K_0]}(\dot \Sigma, M;J), \lambda(dw(z^+)) = 0, \, z^+ \in \dot \Sigma \}
$$
\emph{the moduli space with $\Xi$-tangency condition}.
\end{defn}

The union of  moduli spaces $\widetilde \CM_{(0,1)}^\Xi(\dot \Sigma, M)$ over $J \in=
\CJ_\lambda$ is nothing but
\be
\aleph_1^{-1}(o_{\CC\CD} \times \Xi)/\operatorname{Aut}(\Sigma)
\ee
where $o_{\CC\CD}$ is the zero
section of the bundle $\CC\CD$ defined above, and
$\operatorname{Aut}(\Sigma)$ acts on $((\Sigma,j),w)$ by conformal
reparameterization for any $j$.

The following is the main theorem of this section.

\begin{thm}[$1$-jet evaluation transversality]\label{thm:1-jet}
The derivative
$$
D(\ev^+):T \widetilde \CM_{(0,1)}(\dot \Sigma,M) \to \Lambda^1((\cdot)^*\Xi)
$$
is transverse to $\Lambda^1(w^*\Xi)$, where  $D_{{(J,(j,w),z)} }(\ev^+)$  is given by the linear map
$$
D_{{(J,(j,w),z)} }(\ev^+): T_{(J,(j,w),z)} \widetilde \CM_{(0,1)}(\dot \Sigma,M) \to \Lambda_z^1(w^*TM)
$$
at each $(J,(j,w),z) \in \widetilde \CM_{(0,1)}(\dot \Sigma,M)$.
\end{thm}

We define bundles
\beastar
J^1(\Sigma, M) & = & \bigcup_{(z,x) \in \dot \Sigma \times M}  \Hom(T_z \dot \Sigma , T_xM) \\
J^{\pi 1} (\Sigma, M) & = & \bigcup_{(z,x) \in \dot \Sigma \times M}  \Hom(T_z \dot \Sigma , \Xi_x)  \\
\Omega^i(\Sigma) & = & \bigcup_{z \in \Sigma} \Lambda^i(T_z \dot \Sigma ), \, i = 0, \, 1, \, 2.
\eeastar
Then we have the decomposition
\be\label{eq:dw-decomposition}
J^1(\Sigma, M) = J^{\pi 1}(\Sigma, M) \oplus \CL_\lambda
\ee
where $\CL_\lambda$ is the trivial line bundle given by
\be\label{eq:CLlambda}
\CL_\lambda = \R \langle R_\lambda \rangle.
\ee

When we have a global coordinate as in $\dot \Sigma = \R \times [0,1]$, 
we can also consider a more strict constraint asking the transversality
$$
dU(z) \left(\frac{\del}{\del t}\right)  \in \Xi
$$
and consider the associated moduli space 
\be\label{eq:MM-Xi-t}
\widetilde \CM_{(0,1)}^{\Xi,\del_t} (\R \times [0,1],  M) : = 
\left\{(U,z^+)\, \Big|\, z_+ = \left(\tau^+, \frac12\right), \,  dU(z^+)\left(\frac{\del}{\del t}\right) \in \Xi \right\}.
\ee
In fact this kind of moduli space for the \emph{perturbed contact instanton moduli space}
is used in our proof of Weinstein's conjecture in \cite{oh:weinstein-conjecture}.

\begin{rem}
Note that the virtual dimension of this moduli space has the same dimension as that of 
$\widetilde \CM (\R \times [0,1],  M)$ which is precisely the reason why such a moduli
space is introduced therein.  This may be compared with the real version of the \emph{divisor axiom}
in the Gromov-Witten theory \cite{konts-manin}, but in the one-jet level as in \cite{kim-kresch-oh}.
\end{rem}

\section{Generic 1-jet evaluation transversality: Proof}
\label{sec:1jet-proof}

In this subsection, we give the proof of Theorem \ref{thm:1-jet}.

\subsection{Geometric description of the off-shell framework}

We first provide some informal discussion to motivate the
necessary Fredholm set-up for the study of 1-jet property by adapting the discussion
from \cite{oh:higher-jet} for the case of pseudoholomorphic curves in 
symplectic geometry. The only difference therefrom is the replacement of $TM$ to
$\Xi$ and the addition of the Reeb component 
\be\label{eq:Reeb-component}
d(w^*\lambda \circ j) = 0.
\ee

 We will
provide the precise analytical framework in the end of this discussion.

We consider a triple $(J,(j,w),z)$ of compatible $J$
and $w : (\Sigma,j) \to (M,J)$ a $(j,J)$-holomorphic map and $z \in \Sigma$.
Define a map $\Upsilon$ by
\be
\aleph_1(J,(j,w),z) =\left (D\Upsilon(J,(j,w)),d(\w^*\lambda \circ j); w^*\lambda(z)\right).
\ee
We now identify the domain and the target of the map $\alpha_1$.
For any given $((j,J),z)$, consider the bundles over $\dot \Sigma \times M$
\beastar
\Hom^{\pi(0,1)}_{(j,J)}(T\dot \Sigma,\Xi)&: = & \bigcup_{(z,x)} \Hom^{\pi(0,1)}_{(j_z,J_x)}(T_z \dot \Sigma ,\Xi_x) \\
 \Hom^{\pi(1,0)}(T\dot \Sigma,\Xi_{(j,J)}) &: = & \bigcup_{(z,x)}
\Hom^{\pi(1,0)}_{(j_z,J_x)}(T_z \dot \Sigma ,\Xi_x),
\eeastar
where the above unions are
taken for all $(z,x)$ of $(\S\times M)$. At any $(z,x)$, the
fibers are the $(j,J)$-anti-linear and $(j,J)$-linear parts of
$\Hom(T_z \dot \Sigma ,\Xi_x)$, denoted by $\Hom^{\pi(0,1)}(T_z \S,\Xi_x)$ and
$\Hom^{\pi(1,0)}(T_z \S,\Xi_x)$ respectively. We recall
the decomposition
$$
\Lambda^1(T_x M) = \Lambda^1 (\Xi_x) \oplus \Lambda^1(\CL_{\lambda,x})
\cong \Lambda^1 (\Xi_x) \oplus (\Lambda^1(T\dot \Sigma) \otimes R_\lambda).
$$

For each given $(J,(j,w),z)$, we associate a $2n$-dimensional vector space
$$
 \Lambda_{(j,J)}^{(1,0)}(w^*\Xi)|_{z} =
\Lambda_{(j_z,J_{w(z)})}^{(1,0)}(\Xi_{w(z)})
$$
and define the vector bundle of rank $2n$
$$
 \Lambda^1((\cdot)^*\Xi) \to \CF_1^{\text{\rm univ}}(\dot \Sigma,M)
 $$
 whose fiber at 
 $$
 ((j,w),J,z) \in \CF_1^{\text{\rm univ}}(\dot \Sigma,M)
 $$
 is given by 
 \be\label{eq:Lambda1-Xi}
 \Lambda^1((\cdot)^*\Xi)= \bigcup_{(J,(j,w),z)}\Lambda_{(j,J)}^{(1,0)}(w^*\Xi)|_{z}.
\ee
We denote the corresponding universal moduli space of marked $J$ contact instantons
 $((\dot \Sigma,M),j),w,z)$ by 
 \be\label{eq:universal-CM1}
 \CM_1(\dot \Sigma,M) = \left\{ (J,(j,w),z) \in \CF_1^{\text{\rm univ}}(\dot \Sigma,M)\,
 \Big|\, \delbar^\pi_{(j,J)} w= 0, \, 
 d(w^*\lambda \circ j) = 0 \right\}.
 \ee
We introduce the  bundle over $\CF_1^{\text{\rm univ}}(\dot \Sigma,M)$
$$
\CH^{\pi(0,1)} = \bigcup_{(J,(j,w))} \CH^{\pi(0,1)}_{((j,w),J)}, \quad \CH^{\pi(0,1)}_{(J,(j,w))} = \Omega_{(j,J)}^{(0,1)}(w^*\Xi).
$$
Then we have the expression of 
$$
\aleph_1:  \CF_1^{\text{\rm univ}}(\dot \Sigma,M) \to \CH^{\pi(0,1)}\times
 \Lambda^1((\cdot)^*\Xi)
$$
given by 
\be\label{eq:aleph1}
\aleph_1(J,(j,w),z) = \left(\Upsilon(J,(j,w)), \;(\del_{(j,J)}^\pi w)(z), w^*\lambda(z) \right)
\ee
where $\CH^{\pi(0,1)}\times  \Lambda^1((\cdot)^*\Xi)$ is the fiber product
of the two bundles
$$
\CH^{\pi(0,1)} \to \CF_1^{\text{\rm univ}}(\dot \Sigma,M)
$$
and
$$
 \Lambda^1((\cdot)^*\Xi) 
\to \CF_1^{\text{\rm univ}}(\dot \Sigma,M).
$$
More explicitly we can express the fiber product as
\beastar
&{}& \CH^{\pi(0,1)}\times \Lambda^1((\cdot)^*\Xi)\\
& := & \left\{(\eta, \zeta_0; J,(j,w),z) \, \Big|\,
\eta \in \CH^{\pi(0,1)}_{(J,(j,w))}, \, \zeta_0 \in  \Lambda^1(w^*TM|_z) \right\}
\eeastar
We regard this fiber product as a vector bundle over $\CF_1^{\text{\rm univ}}(\dot \Sigma,M)$,
$$
(\eta, \zeta_0; J,(j,w),z) \mapsto (J,(j,w),z)
$$
whose fiber at $(J,(j,w),z)$ is given by
$$
\CH^{\pi(0,1)}_{(J,(j,w))} \times   \Lambda^1(w^*TM|_z).
$$
Then the above map $\aleph_1$ will become a smooth
\emph{section} of this  vector bundle.

The union of standard moduli spaces $\CM_1(\dot \Sigma, M)$ over $J \in \CJ_\lambda$
is nothing but
\be
\aleph_1^{-1}\left(o_{\CH^{\pi(0,1)}} \times  \Lambda^1((\cdot)^*\Xi)\right)/\operatorname{Aut}(\Sigma)
\ee
where $o_{\CH^{\pi(0,1)}}$ is the
zero section of the bundle $\CH^{\pi(0,1)}$ defined above, and
$\operatorname{Aut}(\Sigma)$ acts on $((\S,j),w)$ by conformal
equivalence for any $j$.  We also denote
 \beastar
\widetilde \CM_{(0,1)}^\Xi(\dot \Sigma, M) & = & \aleph_1^{-1}\left(o_{\CH^{\pi(0,1)}}\right)
\times \Lambda^1((\cdot)^*\Xi))\\
\widetilde \CM_{(0,1)}^{\Xi}(\dot \Sigma, M;J) & = & \widetilde \CM_{(0,1)}^\Xi(\dot \Sigma, M) \cap
\pi_2^{-1}(J). \eeastar

\subsection{Evaluation transversality and its off-shell function spaces}

 At each fixed $(J, (w,j),z_0)$ where we do linearization of $\aleph_1$,
we will write
\beastar
\Omega^0_{k,p}(w^*TM) & := & W^{k,p}(w^*TM) = T_w \CF^{k,p}(\dot \Sigma,M)\\
\Omega^{(0,1)}_{k-1,p}(w^*\Xi) & := & W^{k-1,p}\left(\Lambda_{(j,J)}^{(0,1)}(w^*TM)\right)
\eeastar
for the simplicity of notations. Let $o_{ \Lambda^1((\cdot)^*\Xi)}$ be
the zero section of $ \Lambda^1((\cdot)^*\Xi)$.
We now prove the following proposition by linearizing the section
$\aleph_1$.

\begin{prop}\label{prop:trans-aleph1} The section $\aleph_1$ is transverse to the subset

\be o_{\CH^{\pi(0,1)}}\times \Lambda^1((\cdot)^*\Xi)  \subset \CH^{\pi(0,1)} \times
\Lambda^1((\cdot)^*\Xi)).
\ee
In particular the set
$$
\aleph_1^{-1}\left(o_{\CH^{\pi(0,1)}} \times  \Lambda^1((\cdot)^*\Xi)\right)
$$
is a submanifold of $\widetilde \CM_1(\dot \Sigma, M)$ of codimension $2n$.
\end{prop}

Recall that the subset
\be\label{eq:oCH''}%
o_{\CH^{\pi(0,1)}} \times o_{ \Lambda^1(\cdot)^*\Xi}\subset o_{\CH^{\pi(0,1)}}\times
 \Lambda^1((\cdot)^*\Xi)\ee%
is a submanifold of codimension $2n = \rank \Xi$.

So it is easy to check the statement on the codimension once we
prove $\aleph_1$ is transverse to the submanifold 
$$
 o_{\CH^{\pi(0,1)}}
\times o_{(\cdot)^*\lambda}\subset \CH^{\pi(0,1)} \times  \Lambda^1((\cdot)^*\Xi).
$$
Let
$$
(J,(j,w),z) \in \aleph_1^{-1}(o_{\CH^{\pi(0,1)}} \times o_{H^{(1,0}}).
$$
The linearization of $\aleph_1$ at $(J,(j,w),z)$
$$
D_{(J,(j,w),z)}\aleph_1 : T_J\CJ_\lambda \times T_{((j,w),z)}\CF_1^{k,p}(\dot \Sigma,M) \to
\CH^{\pi(0,1)}_{(J,(j,w))} \times  \Lambda^1(w^*TM|_z)
$$
is given as follows. We decompose $Y = Y^\pi + \lambda(Y) R_\lambda$.
Then 
\be\label{eq:DUpsilon}%
(L, (b,Y),v) \mapsto
\Big(D_{J,(j,w)}\Upsilon(L, (b,Y)),\;D_{J,(j,w)}\del^\pi(L, (b,Y))(z)
+ \nabla_v (\del^\pi_{(j,J)}w)(z) \Big)
\ee%
for 
$$
B\in T_J\CJ_{\o}, \, b\in T_j\CM(\S), \, v\in T_z\S, \, Y \in \Omega^0_{k,p}(w^*TM).
$$
Recall that $w$ is in $W^{k,p}$ with $k \geq 3$. in fact, $w$ is smooth by
elliptic regularity since  $\Upsilon((j,J),w) = 0$ which is equivalent to
$$
\delbar^\pi w = 0, \quad d(w^*\lambda \circ j) = 0.
$$
Furthermore
$D_{J,(j,w)}\del^\pi((L, (b,Y))$ and $\nabla_v (\del^\pi_{(j,J)}w)$ are in
$W^{k-2,p}$ where $k-2 \geq1$. Therefore their evaluations at $z$
are well-defined.

\subsection{Proof of the evaluation transversality against $\Xi$}

Now we are ready to give the proof of Theorem \ref{thm:0jet-evaluation}.
The present subsection will be occupied by the proof thereof.

Let $((\eta, f),X_p) \in \CB^* \times T_pM$ satisfy
\be\label{eq:0}
\left\langle D_w\Upsilon(J,(j,w))Y + \left(\frac{1}{2}L \cdot d^\pi w \circ j,0\right),(\eta,f) \right\rangle
+ \langle Y, \delta_{z_0} X_p \rangle = 0
\ee
for all $Y \in \Omega^0_{2,p}(w^*TM)$ and $L$ where $\delta_{z_0}$ is the
Dirac-delta function supported at $z_0$.
By the Hahn-Banach lemma, it will be
enough to prove
\be\label{eq:vanishing}
(\eta,f) = 0 = X_p.
\ee
In the derivation of \eqref{eq:0},  we have used the formula
$$
\Upsilon(J,(j,w)) = \left(\delbar^\pi_J w, d(w^*\circ \lambda \circ j)\right)
$$
to compute the linearization of $\aleph_1$. In particular, we have
$$
D_1\aleph_1(L)) = \left(\frac{1}{2}L \cdot d^\pi w \circ j,0\right)
$$
in the direction of $J$ where the second factor of the value of $\aleph_1$
does not depend on $J$.  Obviously, we have
$$
D_2\aleph_1(Y) = D\Upsilon(Y).
$$

Without loss of any generality, we
may assume that $Y$ is smooth since the embedding
$$
C^\infty(w^*TM) \hookrightarrow \Omega^0_{2,p}(w^*TM)
$$
has a dense image.
Taking $L=0$ in \eqref{eq:0}, we obtain
\be \label{eq:coker} %
\langle D_w\Upsilon(J,(j,w)) Y, (\eta,f) \rangle + \langle  Y,
\delta_{z_0} X_p \rangle = 0 \quad \mbox{ for all $Y$
of $C^{\infty}$ }.
\ee %
Therefore by definition of the distribution derivatives, $\eta$ satisfies
$$
(D_w\Upsilon(J,(j,w)))^\dagger (\eta,f) - \delta_{z_0} X_p = 0
$$
as a distribution, i.e.,
$$
(D_w\Upsilon(J,(j,w)))^\dagger (\eta,f) = \delta_{z_0} X_p
$$
where $(D_w \Upsilon(J,(j,w)))^\dagger$ is the formal adjoint of $D_w
\Upsilon(J,(j,w))$ whose symbol is of the same type as $D_w \Upsilon_{(j,J)}$ and so is
an elliptic first order differential operator. (See \eqref{eq:matrixDUpsilon} for the linearization formula and
recall that $(\delbar_J^\pi)^\dagger = - \del_J^\pi$ modulo zero order operators.)
By the elliptic regularity, $(\eta,f)$ is a classical solution on $\Sigma \setminus \{z_0\}$.

On the other hand, by setting $Y = 0$ in \eqref{eq:0},  we get
\be
\label{eq:1/2L} \langle L\cdot dw\circ j, (\eta,f) \rangle=0
\ee
for all $L\in T_{J} \CJ_\lambda$. From this identity, the
argument used in the transversality proven in the previous section
shows that $\eta=0$ in a small neighborhood of any somewhere injective point in $\Sigma
\setminus \{\chi^+
\}$. Such a somewhere injective point exists by the
hypothesis of $w$ being somewhere injective and the fact that the
set of somewhere injective points is open and dense in the domain
under the given hypothesis. Then by the unique
continuation theorem, we conclude that $\eta = 0$ on $ \Sigma
\backslash\{z_0\}$ and so the support of $\eta$ as a distribution
on $\Sigma$ is contained at the one-point subset $\{z_0\}$ of $\Sigma$.

The following lemma will conclude the proof of Proposition \ref{prop:trans-aleph1}.
We postpone the proof of the lemma till the last section of the paper.

\begin{lem} \label{lem:eta=01jet}  $\eta$ is a distributional solution of
$(D_w \Upsilon(J,(j,w)))^\dagger (\eta,\omega) = 0$ on $\Sigma$
and so continuous.
In particular, we have $(\eta,\omega) = 0$ in $(\CC\CD)^*$.
\end{lem}

Once we know $(\eta,\omega) = 0$, the equation \eqref{eq:0} is reduced to
the finite dimensional equation
\be\label{eq:simple0}
(Y(z_0),X_p)_{z_0} = 0
\ee
It remains to show that $X_p = 0$. For this, we have only to
show that the image of the evaluation map
$$
Y \mapsto  Y(z_0)
$$
is surjective onto $T_{p}M$, which is now obvious.

We need to prove that at each 
$$
(J,(j,w),z_0) \in
\aleph_1^{-1}\left(o_{\CH^{\pi(0,1)}} \times o_{ \Lambda^1((\cdot)^*\Xi}\right),
$$
the system of
equations, for $Y = \xi +f R_\lambda$ with $Y^\pi =: \xi$ and $\lambda(Y) = f$,
\bea%
D_{J,(j,w)}\delbar^\pi((L, (b,\xi)) &=& \gamma \label{eq:Dbar}\\
D_{J,(j,w)}\del^\pi((L, (b,\xi))(z_0)
+ \nabla_v (\del^\pi_{(j,J)}w)) & = & \zeta_0
\label{eq:DJueta-}
\eea%
has a solution $((L, (b,\xi),v)$ for each given data
$$
\gamma \in \Omega^{(0,1)}_{k-1,p}(w^*\Xi), \quad \zeta_0 \in  \Lambda^1(w^*\Xi)_{z_0}.
$$
It will be enough to consider the triple with $b = 0$ and $v=0$ which we will assume
from now on.

Now we study solvability of \eqref{eq:Dbar}-\eqref{eq:DJueta-} by applying the
Fredholm alternative. For this purpose, we make the following crucial remark

\begin{rem} \label{rem:3pversus2p}
We emphasize  that for the map
$$
(z_0,v) \mapsto D_{J,(j,w)}\del^\pi \left((L, (b,\xi))(z_0)
+ \nabla_v \left(\del^\pi_{(j,J)}w\right)\right)
$$
to be defined as a continuous map to 
$ \Lambda^1(w^*\Xi)_{z_0}$, the map $w$ must
be at least $W^{2 +\epsilon, p}$ for $\e > 0$ : On $W^{2,p}$, the map
$D_{J,(j,w)}\del^\pi(L, (b,\xi))$ will be only in $L^p$ for which
the evaluation at a point is not defined in general, let alone being continuous.
However, the evaluation map
\be\label{eq:W2ptoTM}
z_0 \mapsto D_{J,(j,w)}\del^\pi(L,(0,\xi))(z_0)
\ee
is well-defined and continuous on $W^{2,p}$ as shown by the explicit formula
\eqref{eq:formula}, which involves only \emph{one} derivative of the section $\xi$.
This reduction from $W^{k,p}$ to $W^{2,p}$ of the regularity requirement
in the study of the map \eqref{eq:W2ptoTM}, which can be achieved after restricting to $b=0$, $v = 0$,
will play a crucial role in our proof. See the proof of Lemma \ref{lem:eta=0}.
\end{rem}

Utilizing this remark, we will first show that the image of the map
\eqref{eq:DUpsilon} restricted to the elements of the form
$(L,(0,\xi),0)$ is onto as a map
$$
T_J \CJ_\lambda \times \Omega^0_{2,p}(w^*TM) \to (\CD^{1,p})_{(J,(j,w))}  \times
  \Lambda^1(w^*\Xi)_{z_0}
$$
where $(w,j,z_0,J)$ lies in $o_{\CH^{(0,1)}}\times o_{ \Lambda^1((\cdot)^*\Xi)}$. In the end of
the proof, we will establish solvability of \eqref{eq:Dbar}-\eqref{eq:DJueta-}
on $W^{k,p}$ for $\gamma \in W^{k-1,p}$ by applying an elliptic regularity result
of the map \eqref{eq:Dbar}.

We regard
$$
\Omega^{(0,1)}_{1,p}(w^*\Xi) \times   \Lambda^1(w^*\Xi)_{z_0}: = \CB
$$
as a Banach space with the norm
$$
\|\cdot\|_{1,p} + |\cdot|
$$
where $|\cdot|$ any norm induced by an inner product on
$$
H^{(1,0
)}_{(J,(j,w),z_0)}=\Lambda^{(1,0)}_{(j,J)}(w^*\Xi)_{z_0} \cong \C^n.
$$
For the clarification of notations, we denote the natural pairing
$$
\Omega^{(0,1)}_{1,p}(w^*\Xi)\times \left(\Omega^{(0,1)}_{1,p}(w^*\Xi)\right)^*\to \R
$$
by $\langle \cdot, \cdot
\rangle$ and the inner product on $ \Lambda^1((\cdot)^*\Xi_{(J,(j,w),z_0)}$ by $(\cdot,
\cdot)_{z_0}$.

Finally we have the natural projection
$$
\Pi:\widetilde \CM_1(\dot \Sigma, M) : = \bigcup_{J \in \CJ_\lambda}
\widetilde \CM_1(\dot \Sigma, M) \to \CJ_\lambda.
$$
The projection has index $2(c_1(\beta) + n(1-g)) + 2$, so for any
regular value $J$. 

$\Aut(\Sigma)$ acts on marked Riemann surfaces $((\S,j),z)$ by
conformal equivalence then on the maps from them.
We define the moduli space
$$
\CM_1^{\text{\rm crit}}(M,J;\b):=\widetilde\CM_1^{\text{\rm crit}}(M,J;\b)/\Aut(\Sigma),
$$
where 
$\CM_1^{\text{\rm crit}}(M,J;\b)$ consists of  $J$ contact instantons in class
$\b$ with at least one critical point. As a smooth orbifold, we have
$$
\text{dim }\CM_1^{\text{\rm crit}}(M,J;\b)=2 \big(c_1(\b)+(3-n)(g-1)+1-n\big)
$$
for all $g$ as stated in the introduction.

Finally we just set
$$
\CJ_\lambda^{\Xi} := \mbox{the set of regular values of $\Pi$}
$$
which finishes the proof of the following theorem.

\begin{thm} \label{thm:immersion} There exists a subset $\CJ_\lambda^{\Xi} \subset
\CJ_\lambda$ of second category such that for $J \in
\CJ_\lambda^{\Xi}$ all $J$ contact instanton
$w:\Sigma \to M$ are immersed for any $j \in \CM(\Sigma)$, provided
\be\label{eq:immersed} c_1(\beta) + (3-n)(g-1)< n -1 \ee
\end{thm}

Now it remains to prove Lemma \ref{lem:eta=0} which is postponed.

\section{Proof of Lemma \ref{lem:eta=01jet}.}

Our primary goal is to prove
\be\label{eq:tildexi} \langle D_w\Upsilon(J,(j,w)) Y, (\eta,\omega) \rangle = 0
\ee
for all smooth $Y \in \Omega^0(w^*TM)$, i.e., $\eta$ is a distributional solution of
$$
(D_w\Upsilon(J,(j,w)))^\dagger (\eta,\omega) = 0
$$
\emph{on the whole $\Sigma$}, not just on $\Sigma \setminus \{z_0\}$. This will imply that
$(\eta,\omega)$ is a solution smooth everywhere by the elliptic regularity.

We start with \eqref{eq:coker}
\be\label{eq:coker-append} \langle
D_w\Upsilon(J,(j,w)) Y, (\eta,\omega) \rangle + \langle Y,
\delta_{z_0}X_p \rangle = 0 \quad \mbox{for all $Y \in
C^{\infty}$}.
\ee
We first simplify the expression of the
pairing $\langle D_w\Upsilon(J,(j,w)) Y, (\eta,f) \rangle$ knowing that
$\supp (\eta,f) \subset \{z_0\}$.

Let $z$ be a complex coordinate centered at a fixed marked point $z_0$
and
$$
(x_1,y_1,x_2,y_2,\cdots,x_n,y_n, \eta)
$$
be a Darboux coordinates so that $\lambda = d\eta - \sum_{i=1}y_i dx_i$
on a neighborhood of $p \in M$.
\begin{rem} In the proof of \cite[Section 10.5]{oh:book1}, we chose a complex coordinates
$(w_1, \cdot, w_n)$ identifying a neighborhood of $p$ with an open subset of $\C^n$.
\end{rem}

We consider the standard metric
$$
h = \frac{\sqrt{-1}}{2} dz d\bar z
$$
on a neighborhood $U \subset \dot \Sigma$ of $z_0$.

The following lemma will be crucial in our proof.
\begin{lem}\label{lem:dJbar=dbar} Let $\eta$ be as above.
For any smooth section $Y$ of $w^*(TM)$ and $\eta$ of
$\left(\Omega^{(0,1)}_{1,p}(w^*\Xi)\right)^*$
$$
\langle D\delbar_J^\pi(Y), \eta \rangle = \langle \delbar Y^\pi, \eta \rangle
$$
where $\delbar$ is the standard Cauchy-Riemann operators on $\C^n$ in the above coordinate.
\end{lem}
\begin{proof} We have already shown that
$\eta$ is a distribution with $\supp (\eta,f) \subset \{z_0\}$. By the
structure theorem on the distribution supported at a point $z_0$ Lemma \ref{lem:gelfand}, we have
$$
\eta = P\left(\frac{\del}{\del s}, \frac{\del}{\del t}\right)(\delta_{z_0})
$$
where $z = s + it$ is the given complex coordinates at $z_0$ and
$P\left(\frac{\del}{\del s}, \frac{\del}{\del t}\right)$ is a differential
operator associated by the polynomial $P$ of two variables with coefficients
in
$$
\Lambda^{(0,1)}_{(j_{z_0},J_p)}(w^*\Xi).
$$

Furthermore since $\eta \in (W^{1,p})^*\cong W^{-1,q}$,
the degree of $P$ \emph{must be zero}
and so we obtain
\be\label{eq:eta=adelta}
\eta = \beta_{z_0}\cdot \delta_{z_0}
\ee
for some constant vector $\beta_{z_0} \in \Lambda^{(0,1)}(\Xi_p) $.

We then have the expression
$$
D\delbar_J^\pi Y = \delbar Y + E \cdot \del Y + F\cdot Y
$$
near $z_0$ in coordinates  where $E$ and $F$ are zero-order matrix operators
with $E(z_0) = 0 = F(z_0)$.
(See \cite[p.331]{oh-zhu:ajm} and \cite{sikorav:holo} for such a derivation.)
Therefore by \eqref{eq:eta=adelta}, we derive
\beastar
&{}& \langle E \cdot \del^\pi Y^\pi + F \cdot Y^\pi, (\eta,f) \rangle  =  \langle E \cdot \del^\pi Y^\pi
+ F\cdot Y^\pi, \beta_{z_0} \delta_{z_0} \rangle\\
&{}& \quad  = (E(z_0) \del^\pi Y^\pi(z_0) + F(z_0) Y^\pi(z_0), \beta_{z_0})_{z_0} = 0.
\eeastar
Therefore we obtain
$$
\langle D_w \delbar_J(Y), \eta \rangle = \langle \delbar^\pi Y^\pi + E \cdot \del^\pi Y^\pi + F\cdot Y^\pi,
\eta \rangle = \langle \delbar^\pi Y^\pi, \eta\rangle
$$
which finishes the proof. \end{proof}

By this lemma, \eqref{eq:coker-append} becomes
\be\label{eq:coker-simple}
\langle \delbar^\pi Y^\pi, \eta \rangle + \langle -\Delta (\lambda(Y)) dA
+ d((Y^\pi \rfloor d\lambda) \circ j, f \rangle
+ \langle Y, \delta_{z_0} X_p \rangle = 0
\ee
for all $Y$. We rewrite the middle summand by direct calculation.

\begin{lem}\label{lem:mid-summand} We have
$$
 \langle -\Delta (\lambda(Y)) dA
+ d((Y^\pi \rfloor d\lambda) \circ j, f \rangle
 = - \int \lambda(Y)\, \Delta f \, dA + \int df \circ j \wedge (Y^\pi \rfloor d\lambda)
 $$
 \end{lem}
Noting that $\lambda(Y)$ and $Y^\pi$ are independent, we rearrange the summand of
\eqref{eq:coker-simple} into
\beastar
 0 &= & - \int \lambda(Y)\, \Delta f \, dA  \\
  &{}& + \int \langle  \delbar^\pi Y^\pi, \eta \rangle
  + \int df \circ j \wedge (Y^\pi \rfloor d\lambda) \\
  &{}& + \langle Y^\pi + \lambda(Y)\, R_\lambda, \delta_{z_0} X_p \rangle
\eeastar
Therefore we have derived
\bea
0 &= &- \int \lambda(Y)\, \Delta f \, dA +\langle  \lambda(Y)\, R_\lambda, \delta_{z_0} X_p \rangle
\label{eq:Reeb-term}\\
0 & = & \int \langle  \delbar^\pi Y^\pi, \eta \rangle
  + \int df \circ j \wedge (Y^\pi \rfloor d\lambda)
  + \langle Y^\pi , \delta_{z_0} X_p \rangle  \label{eq:Xi-term}.
 \eea
 We decompose $Y$ as
$$
Y(z) = (Y(z) - \chi(z) Y(z_0)) + \chi(z) Y(z_0)
$$
on $U$ where $\chi$ is a cut-off function with $\chi \equiv 1$ in a
small neighborhood $V \subset U$ of $z_0$ and satisfies $\supp \chi \subset
U$. It induces the corresponding decomposition of $Y^\pi$ and $\lambda(Y)$.

We examine \eqref{eq:Xi-term}.
Then the first summand $\widetilde Y^{\pi}$ defined by
$$
\widetilde Y^{\pi}(z):= Y^{\pi}(z) - \chi(z) Y^{\pi}(z_0)
$$
is a smooth section on $\Sigma$, and satisfies
$$
\widetilde Y^{\pi}(z_0) = 0, \quad \delbar \widetilde Y^{\pi} = \delbar{Y^{\pi}} \quad \mbox{on $V$}
$$
since $\chi(z) Y^{\pi}(z_0) \equiv Y^{\pi}(z_0)$ on $V$.
Therefore applying \eqref{eq:Xi-term} to $\widetilde Y^{\pi}$ instead of $Y^{\pi}$
and recalling $\supp (\eta,f) \subset \{p\}$, we obtain
$$
\langle \delbar \widetilde Y^{\pi}, (\eta,f) \rangle + \langle \widetilde Y^{\pi}, \delta_{z_0} X_p
\rangle = 0.
$$
Again using the support property $\supp \eta \subset \{z_0\}$
and \eqref{eq:coker-append},  we derive
\be\label{eq:delxi-adz}
\langle \widetilde Y^{\pi}, \delta_{z_0}X_p \rangle  =
\langle \widetilde Y^{\pi}(z_0), X_p\rangle= 0
\ee
and so $\langle \delbar \widetilde Y^{\pi}, (\eta,f) \rangle = 0$.
But we also have
\be\label{eq:tildexi=xi}
\langle \delbar^\pi Y^\pi, \eta \rangle=\langle \delbar^\pi \widetilde Y^\pi, \eta \rangle
\ee
since $\delbar^\pi \widetilde Y^\pi = \delbar{Y^\pi}$ on $V$ and $\supp \eta \subset \{z_0\}$.
Hence we obtain
$$
\langle \delbar^\pi_J Y^\pi, \eta \rangle = 0
$$
for all $Y^\pi$.
Applying similar reasoning to \eqref{eq:Reeb-term}, we have
derived
$$
\int \lambda(Y) \Delta f\, dA = 0
$$
for all $\lambda(Y)$. Combining the two, we have proved $(\eta,f)$ are weak solutions
$$
(\delbar^\pi_J)^*\eta = 0 = \Delta f
$$
on whose $\Sigma$. Therefore a we have finished the proof of \eqref{eq:tildexi} by Lemma \ref{lem:dJbar=dbar}.
By the elliptic regularity, $(\eta,f)$ is a smooth solution. In particular
it is continuous.  Since we have
already shown $(\eta,f) = 0$ on $\Sigma \setminus \{z_0\}$, continuity
of $\eta$ proves $(\eta,f) = 0$ on the whole $\Sigma$.
This finishes the proof.

This in turn finishes the proof of Lemma  \ref{lem:eta=01jet}.

\bibliographystyle{amsalpha}

\bibliography{biblio2}

\end{document}